\documentclass[11pt]{article}
\usepackage[a4paper, total={6in, 8in}]{geometry}
\usepackage{enumerate}
\usepackage{amsmath}
\usepackage{amssymb,latexsym}
\usepackage{amsthm}
\usepackage{color}
\usepackage{cancel}
\usepackage{graphicx}

\usepackage{hyperref}
\usepackage{comment}
\usepackage{amscd}

\DeclareMathOperator{\C}{\mathcal{C}}

\newtheorem{theorem}{Theorem}[section]

\newtheorem{lemma}[theorem]{Lemma}
\newtheorem{corollary}[theorem]{Corollary}
\newtheorem{definition}[theorem]{Definition}
\newtheorem{proposition}[theorem]{Proposition}

\newtheorem{remark}[theorem]{Remark}

\newtheorem{open}[theorem]{Open Problem}

\newcommand{\fqn}{\mathbb{F}_{q^n}}

\newcommand{\F}{{\mathbb F}}
\newcommand{\KK}{{\mathbb K}}

\newcommand{\fq}{{\mathbb F}_{q}}

\newcommand{\la}{\langle}
\newcommand{\ra}{\rangle}

\newcommand{\PG}{\mathrm{PG}}
\newcommand{\N}{\mathrm{N}}

\title{$r$-fat linearized polynomials over finite fields}
\author{Daniele Bartoli\thanks{Dipartimento di Matematica e Informatica, Universit\`a degli studi di Perugia,  Perugia, Italy. daniele.bartoli@unipg.it}, 
Giacomo Micheli\thanks{Department of Mathematics \& Statistics, University of South Florida, Tampa, United States. gmicheli@usf.edu},
Giovanni Zini\thanks{Dipartimento di Matematica e Fisica, Universit\`a degli Studi della Campania ``Luigi Vanvitelli'', Caserta, Italy.
giovanni.zini@unicampania.it}, and 
Ferdinando Zullo\thanks{Dipartimento di Matematica e Fisica, Universit\`a degli Studi della Campania ``Luigi Vanvitelli'', Caserta, Italy.
ferdinando.zullo@unicampania.it}
}
\date{ }

\begin{document}

\maketitle
\begin{abstract}
In this paper we prove that the property of being scattered for a $\mathbb{F}_q$-linearized polynomial of small $q$-degree over a finite field $\mathbb{F}_{q^n}$ is unstable, in the sense that, whenever the corresponding linear set has at least one point of weight larger than one, the polynomial is far from being scattered.
To this aim, we define and investigate $r$-fat polynomials, a natural generalization of scattered polynomials.
An $r$-fat $\mathbb{F}_q$-linearized polynomial defines a linear set of rank $n$ in the projective line of order $q^n$ with $r$ points of weight larger than one.
When $r$ equals $1$, the corresponding linear sets are called clubs, and they are related with a number of remarkable mathematical objects like KM-arcs, group divisible designs and rank metric codes.
Using techniques on algebraic curves and global function fields, we obtain numerical bounds for $r$ and the non-existence of exceptional $r$-fat polynomials with $r>0$.
In the case $n\leq 4$, we completely determine the spectrum of values of $r$ for which an $r$-fat polynomial exists.
In the case $n=5$, we provide a new family of $1$-fat polynomials.
Furthermore, we determine the values of $r$ for which the so-called LP-polynomials are $r$-fat.
\end{abstract}



\section{Introduction}

A polynomial of the shape $f(x)=\sum_{i=0}^k a_i x^{q^i}$ over the finite field $\fqn$ of order $q^n$ is called an $\fq$-linearized polynomial, or a $q$-polynomial, or simply a linearized polynomial if $q$ is clear from the context.
%An $\fq$-linearized polynomial (or a $q$-polynomial, or a linearized polynomial if $q$ is clear from the context) over $\fqn$ is any element of the shape $f(x)=\sum_{i=0}^{k} a_i x^{q^i} \in \fqn[x]$.
We identify $f(x)$ with the $\fq$-linear map $x\mapsto f(x)$ over $\fqn$, and we assume the $q$-degree of $f(x)$ to be smaller than $n$; in this way, $\fq$-linearized polynomials over $\fqn$ are in one-to-one correspondence with $\fq$-linear maps over $\fqn$.

An $\fq$-linearized polynomial $f(x)\in \fqn[x]$ is said to be \emph{scattered} if \[ \dim_{\fq} \ker(f(x)-mx) \leq 1, \]
for every $m \in \fqn$.
The concept of scattered polynomial was introduced in \cite{John} and slightly generalized in \cite{BZ}.

\begin{definition}{\rm \cite{BZ,John}}
An $\fq$-linearized polynomial $f(x)\in \mathbb{F}_{q^n}[x]$ is called a scattered polynomial of index $t\in\{0,\ldots,n-1\}$ if 
\[ \dim_{\fq} \ker(f(x)-mx^{q^t}) \leq 1, \]
for every $m \in \fqn$.
Also, a scattered polynomial of index $t$ is exceptional if it is scattered of index $t$ over infinitely many extensions $\mathbb{F}_{q^{nm}}$ of $\mathbb{F}_{q^n}$.
\end{definition}

In this paper, we investigate a more general family of polynomials.

\begin{definition}
An $\fq$-linearized polynomial $f(x)\in \mathbb{F}_{q^n}[x]$ is called an $r$-fat polynomial of index $t\in\{0,\ldots,n-1\}$ if there exist exactly $r$ elements $m_1,\ldots,m_r \in \fqn$ such that
\[ \dim_{\fq} \ker (f(x)-m_ix^{q^t})>1 \]
for each $i \in \{1,\ldots,r\}$.
We say that $f(x)$ has maximum weight $M$ if
\[
\max\{\dim_{\fq} \ker (f(x)-m_ix^{q^t})\colon i \in \{1,\ldots,r\}\}=M.
\]
Also, an  $r$-fat polynomial of index $t$ is exceptional if it is $r$-fat of index $t$ over infinitely many extensions $\mathbb{F}_{q^{nm}}$ of $\mathbb{F}_{q^n}$.
\end{definition}
Note that $0$-fat means scattered.  Up to now, the only known exceptional $r$-fat polynomials are scattered. 
%Also, we call a $1$-fat polynomial  an \emph{$i$-club polynomial} if it has maximum weight $i$.

Scattered polynomials were introduced to study  \emph{scattered linear sets}. Any $q$-polynomial over $\fqn$ defines a so-called $\fq$-linear set 
\begin{equation}\label{eq:defLft}
L_{f,t}=\{\langle (x^{q^t},f(x)) \rangle_{\fqn} \colon x \in \fqn^*\}\subset \PG(1,q^n)
\end{equation}
of rank $n$ in $\PG(1,q^n)$, for any $t\geq0$.
Conversely, any $\fq$-linear set of rank $n$ in ${\rm PG}(1,q^n)$ not containing the point $\langle(0,1)\rangle_{\fqn}$ (which we can always assume after a suitable projectivity) is of type $L_{f,t}$ for some $q$-polynomial $f(x)\in\fqn[x]$ and $t\geq0$.
%We also point out that if the point $\langle (0,1) \rangle_{\F_{q^n}}$ is not contained in a linear set  of rank $n$ of $\PG(1,q^n)$ (which we can always assume after a suitable projectivity), then it is of type $L_{f,t}$ for some $q$-polynomial $\displaystyle f(x)=\sum_{i=0}^{n-1}a_ix^{q^i}\in \fqn [x]$ and a positive integer $t$.

%If $t=0$, then we will use $L_f$ instead of $L_{f,0}$.
More generally, let $\Lambda=\PG(V,\F_{q^n})=\PG(r-1,q^n)$, where $V$ is a vector space of dimension $r$ over $\F_{q^n}$.
A point set $L$ of $\Lambda$ is said to be an \emph{$\F_q$-linear set} of rank $h$ in $\Lambda$ if it is
defined by the non-zero vectors of a $h$-dimensional $\F_q$-vector subspace $U$ of $V$, i.e.
\[L=L_U=\{\la {\bf u} \ra_{\mathbb{F}_{q^n}} \colon {\bf u}\in U\setminus \{{\bf 0} \}\}.\]
The \emph{weight} in $L_U$ of a point $P=\langle \mathbf{u} \rangle_{\F_{q^n}}$ is $w_{L_U}(P)=\dim_{\F_q}(U\cap\langle \mathbf{u} \rangle_{\F_{q^n}})$.

%Denote by $N_i$ the number of points of $\Lambda$ having weight $i\in \{0,\ldots,n\}$ in $L_U$. The following relations hold:
%\begin{equation}\label{eq:card}
%    |L_U| \leq \frac{q^k-1}{q-1},
%\end{equation}
%\begin{equation}\label{eq:pesicard}
%    |L_U| =N_1+\ldots+N_k,
%\end{equation}
%\begin{equation}\label{eq:pesivett}
%    N_1+N_2(q+1)+\ldots+N_k(q^{k-1}+\ldots+q+1)=q^{k-1}+\ldots+q+1.
%\end{equation}

%After a suitable projectivity, every $\fq$-linear set of rank $n$ in $\PG(1,q^n)$ is equal to $L_{f,t}$ as defined in \eqref{eq:Lft}, for some linearized polynomial $f \in \fqn[x]$ and some positive integer $t$.

Two linear sets $L_{U_1}$ and $L_{U_2}$ in $\Lambda$ are $\mathrm{P}\Gamma \mathrm{L}$-\emph{equivalent} if there exists $\varphi \in \mathrm{P}\Gamma \mathrm{L} (r-1,q^n)$ such that $\varphi(L_{U_1})=L_{U_2}$.
We will say that two $q$-polynomials $f_1(x)$ and $f_2(x)$ over $\fqn$ are \emph{equivalent} if $L_{f_1,t_1}$ and $L_{f_2,t_2}$ are $\mathrm{P}\Gamma \mathrm{L}$-equivalent, for some $t_1,t_2 \in \{0,\ldots,n-1\}$.

A $q$-polynomial $f(x) \in \fqn[x]$ is scattered of index $t$ if and only if $L_{f,t}$ has $\frac{q^n-1}{q-1}$ points, i.e. $L_{f,t}$ is \emph{scattered}.
Recently, scattered polynomials and scattered linear sets played a central role in the theory of rank metric codes;
see \cite{Sheekeysurvey,PZScatt,ZiniZulloScatt} for details on the connection between scattered linear sets and MRD codes.
%Furthermore, if $f(x)$ is an $r$-fat polynomial of index $t$ if and only if $L_{f,t}$ has exactly $r$ distinct points having weight in $L_{f,t}$ greater than one. 

$r$-fats polynomials are related to a special type of linear sets when $r=1$.
In fact, $1$-fat polynomials define \emph{$i$-clubs} of ${\rm PG}(1,q^n)$, i.e. $\fq$-linear sets in which one point has weight $i$ and all the others have weight one.
%An $i$-\emph{club} of rank $n$ in $\PG(1,q^n)$ is an $\F_{q}$-linear set in $\PG(1,q^n)$ such that one point has weight $i$ and all the others have weight one.
%By \eqref{eq:pesicard} and \eqref{eq:pesivett}, we have that an $i$-club has size $q^{n-1}+\ldots+q^i+1$. 
As we will recall later, the relevance of clubs arises in the study of a special type of arcs in desarguesian projective planes of even order, called KM-arcs. Clubs are also related with group divisible designs which are embeddable in a finite projective plane; see \cite{CW}.
It is worth noting that $1$-fat polynomials provide three-weight rank metric codes, see \cite{John}.

Inspired by the classification results for exceptional scattered polynomials contained in \cite{BM,BZ,FM2020}, we use machineries both from algebraic geometry over finite fields and algebraic number theory to investigate the existence of $r$-fat polynomials.

The connection between algebraic curves and $r$-fat polynomials follows directly from the one between curves and scattered polynomials described in \cite{BM,BZ}.
For any $r$-fat polynomial $f(x)$ of index $t$ over $\fqn$, consider the algebraic plane curves
\[\mathcal{N}\colon f(X) Y^{q^t}-f(Y)X^{q^t}=0,\qquad\mathcal{C}\colon \frac{f(X) Y^{q^t}-f(Y)X^{q^t}}{X^qY-XY^q}=0,\]
whose equations depend on the polynomial and its index. We provide the existence of an $\fqn$-rational absolutely irreducible component of $\C$ and then we apply the Hasse-Weil bound in order to get a bound on the weight distribution of the linear set $L_{f,t}$, which also provides a bound on the integer $r$.
Another technique from \cite{FM2020} based on Galois-theoretical tools is applied here to get different bounds on the weight distribution of $L_{f,t}$ and on the integer $r$.
These bounds imply the non-existence of exceptional $r$-fat polynomials when $r>0$. 

For $n=4$, by using the classification of linear blocking sets in $\mathrm{PG}(2,q^4)$ obtained in \cite{BoPol}, we classify the $\fq$-linear sets of rank $4$ in $\PG(1,q^4)$. 
As a consequence, we determine all the values of $r$ for which there exists at least one $r$-fat polynomial over $\mathbb{F}_{q^4}$.
For $n=5$, constructions of clubs in $\PG(1,q^5)$ which are not defined by the trace function were already known, but a $1$-fat polynomial representation for them is unknown.
We provide a new example of $1$-fat polynomial over $\mathbb{F}_{q^5}$ which is not equivalent to the trace function (the only previously known one), and we also determine the stabilizer in $\mathrm{P\Gamma L}(2,q^5)$ of the associated linear set. 
Finally, for any $n$, we find the complete weight distribution of the family of linear sets of LP-type named after Lunardon and Polverino, namely $L_{f,s}$ where $f(x)=x+\delta x^{q^{2s}}$ and $\gcd(s,n)=1$,
whose points have weight at most two. Indeed, by investigating a certain quadratic form over $\fqn$, we are able to count the points having weight two. This also provides examples of $r$-flat polynomials over $\fqn$ for $r\in \left\{\frac{q^{n-1}-1}{q^2-1},\frac{q^2(q^{\frac{n-2}2}+1)(q^{\frac{n-2}2-1}-1)}{q^2-1}+1,\frac{(q^{n/2}+1)(q^{n/2-1}-1)}{q^2-1}\right\}$.

The paper is organized as follows.
In Section \ref{sec:pre} we briefly describe some tools from algebraic geometry, algebraic number theory, and linear sets theory.
Section \ref{sec:difficile} is devoted to the the study of $r$-fat polynomials through the investigation of certain algebraic curves over finite fields and Galois extensions of algebraic function fields over finite fields.
In Section \ref{sec:known} we resume what is known about $1$-fat polynomials and clubs of rank $n$ in $\PG(1,q^n)$.
As a consequence of the classification of $\fq$-linear blocking sets in $\PG(2,q^4)$, we classify in Section \ref{sec:q4} the $\fq$-linear sets of rank $4$ in $\PG(1,q^4)$. 
As a byproduct, we determine all the values of $r$ for which there exists at least one $r$-fat polynomial over $\mathbb{F}_{q^4}$.
In Section \ref{sec:q5} we give an example of a $1$-fat polynomial over $\mathbb{F}_{q^5}$ not equivalent to the trace function and we study the associated linear set.
Section \ref{sec:LP} is devoted to the study of the linear sets of LP-type and the determination of their weight distribution.
Section \ref{sec:open} concludes the paper with some open problems that could be of interest for further research.

\section{Preliminaries}\label{sec:pre}

\subsection{Algebraic curves and function fields over finite fields}

We start by recalling some elementary facts about plane curves and a simplified version of the Hasse-Weil bound; for a detailed exposition we refer to \cite{HKT}.

Let $\KK$ be a field, $\overline{\KK}$ be an algebraic closure of $\KK$, $f(X,Y)\in\overline{\KK}[X,Y]$ be a bivariate polynomial over $\overline{\KK}$ of degree $d\geq1$, and $\mathcal{F}\subset\PG(2,\overline{\KK})$ be the projective algebraic plane curve with affine equation $\mathcal{F}\colon f(X,Y)=0$.
Equivalently, $\mathcal{F}$ has homogeneous equation $f_h(X,Y,Z)=0$, where $f_h(X,Y,Z)=Z^d f(X/Z,Y/Z)$.
The degree $\deg(\mathcal{F})\geq1$ is defined as the degree of the polynomial $f(X,Y)$.
For any line $\ell\subset\PG(2,\overline{\KK})$, $\deg(\mathcal{F})$ is an upper bound on the number of intersection points between $\mathcal{F}$ and $\ell$.
A point $P$ of $\mathcal{F}$ is singular if and only if the three partial derivatives of $f_h(X,Y,Z)$ vanish at $P$; a curve with no singular points is called non-singular.
If a non-singular curve $\mathcal{G}\subset\PG(2,\overline{\KK})$ is a non-repeated component of $\mathcal{F}$ defined by a divisor $g(X,Y)$ of $f(X,Y)$ and $P$ is a point of $\mathcal{G}$, then $P$ is singular for $\mathcal{F}$ if and only if $P$ is a point of the component of $\mathcal{F}$ with affine equation $f(X,Y)/g(X,Y) =0$.  

Suppose that $\lambda\cdot f(X,Y)\in\KK[X,Y]$ for some $\lambda\in\overline{\KK}^*$. Then the curve $\mathcal{F}$ is said to be defined over $\KK$, or $\KK$-rational, and we denote by $\mathcal{F}(\KK)=\mathcal{F}\cap\PG(2,\KK)$ the set of its $\KK$-rational points.
A curve defined by a non-constant divisor of $f(X,Y)$ is called a component of $\mathcal{F}$. We call $\mathcal{F}$ reducible or irreducible over $\KK$ according respectively to $f(X,Y)$ being reducible or irreducible over $\KK$; we call $\mathcal{F}$ absolutely irreducible if $f(X,Y)$ is irreducible over $\overline{\KK}$.
The number of points of absolutely irreducible curves defined over a finite field is bounded by the following simplified version of the Hasse-Weil bound.
\begin{theorem}\label{th:hasse-weil}{\rm (Hasse-Weil bound, see \cite[Equation 9.46]{HKT})}
Let $\mathcal{F}$ be an absolutely irreducible plane curve of degree $d$ defined over a finite field $\fqn$. Then
\[ |\#\mathcal{F}(\fqn)-(q^n+1)|\leq (d-1)(d-2)\sqrt{q^n}. \]
\end{theorem}

We now recall a Galois-theoretic tool about global function fields, which are the algebraic counterpart of algebraic curves over a finite field; for more details we refer to \cite{Sti}.

Let $F/\fqn$ be a global function field with full constant field $\fqn$, that is, $F$ is an extension field of $\fqn$ of transcendence degree $1$ such that any element of $F\setminus\fqn$ is transcendent over $\fqn$.
Let $M$ be a finite Galois extension of $F$, $K=\overline{\fq}\cap M$ be the full constant field of the function field $M/\fqn$, $g_M$ be the genus of $M$, and $K F$ denote the compositum of $K$ and $F$.
The Galois group $G^{\rm arith}={\rm Gal}(M/F)$ is also called the arithmetic Galois group of $M/F$, while the Galois group $G^{\rm geom}={\rm Gal}(M/K F)$ of $M/K F$ is called the geometric Galois group of $M/F$.
Let $\sigma\in G^{\rm arith}$ be an element acting on $K$ as the Frobenius automorphism of $K/\fqn$, i.e. $\phi:=\sigma\!\mid_K\colon \lambda\mapsto \lambda^{q^n}$, and let $\Gamma_{\sigma}\subseteq G^{\rm arith}$ be the conjugacy class of $\sigma$.
Let $\mathcal{P}$ be a place of $F$ of degree $1$ with valuation ring 
$\mathcal{O}_{\mathcal{P}}$ and $\mathcal{R}$ be a place of $M$ lying over $P$ with valuation ring $\mathcal{O}_{\mathcal{R}}$. Then $\mathcal{O}_{\mathcal{P}}/\mathcal{P}$
 and $\mathcal{O}_{\mathcal{R}}/\mathcal{R}$ are naturally isomorphic respectively to 
 $\fqn$ and to an intermediate field of $K/\fqn$, and hence 
 $\phi\in{\rm Gal}\left(\mathcal{O}_{\mathcal{R}}/\mathcal{R}\colon\mathcal{O}_{\mathcal{P}}/\mathcal{P}\right)$.
Since the elements of the decomposition group $D(\mathcal{R}|\mathcal{P})=\{\sigma \in G^\mathrm{arith} \colon \sigma (\mathcal{R})=\mathcal{R}\}$ naturally act as automorphisms of $\mathcal{O}_{\mathcal{R}}/\mathcal{R}$, we can consider the subgroup $D_{\phi}(\mathcal{R}|\mathcal{P})$ of those elements of $D(\mathcal{R}|\mathcal{P})$ acting on $\mathcal{O}_{\mathcal{R}}/\mathcal{R}$ as $\phi\in{\rm Gal}\left(\mathcal{O}_{\mathcal{R}}/\mathcal{R}\colon\mathcal{O}_{\mathcal{P}}/\mathcal{P}\right)$, and define
\[
w_{\mathcal{P}}(\sigma)=\frac{\#(D_{\phi}(\mathcal{R}|\mathcal{P})\cap\Gamma_{\sigma})}{\#D_{\phi}(\mathcal{R}|\mathcal{P})\cdot\#\Gamma_{\sigma}}\in[0,1].
\]
Denote by $P^1(F)$ the set of places of $F/\fqn$ of degree $1$. Notice that, if $\mathcal{P}\in P^1(F)$ is unramified under $M$, then either $D_{\phi}(\mathcal{R}|\mathcal{P})\subseteq\Gamma_{\sigma}$ or $D_{\phi}(\mathcal{R}|\mathcal{P})\cap\Gamma_{\sigma}=\emptyset$ according respectively to $\sigma$ being a Frobenius at $\mathcal{P}$ or not.
Therefore
\[
\sum_{\mathcal{P}\in P^1(F)} w_{\mathcal{P}}(\sigma)= \sum_{{\rm ramified }\;\mathcal{P}\in P^1(F)} w_{\mathcal{P}}(\sigma) + \sum_{\substack{{\rm unramified}\;\mathcal{P}\in P^1(F),\\
\sigma\;{\rm is\; a\; Frobenius\; at}\;\mathcal{P}
}} \frac{1}{\#\Gamma_{\sigma}}
\]
and hence Chebotarev density theorem can be stated as follows.

\begin{theorem}{\rm (Chebotarev density theorem, see \cite[Remark 2.3]{FM2020b})}\label{th:chebotarev}
With the same notation as above, denote by $U$ be the number of places of $F/\fqn$ of degree $1$ at which $\sigma$ is a Frobenius and which are unramified under $M$. Then
\[
\Bigg| U + \sum_{{\rm ramified }\;\mathcal{P}\in P^1(F)} \#\Gamma_{\sigma}\cdot w_{\mathcal{P}}(\sigma)  - \frac{\#\Gamma_{\sigma}}{\# G^{\rm geom}}\cdot(q^n+1) \Bigg|  \leq \frac{2\cdot\#\Gamma_{\sigma}}{\# G^{\rm geom}}\cdot g_M\cdot\sqrt{q^n}.
\]
\end{theorem}

The following lemma is a classical fact from algebraic number theory, whose proof can be found for example in \cite{guralnick2007exceptional}.

\begin{lemma}\label{orbits}
Let $L/F$ be a finite separable extension of global function fields, $M$ be its Galois closure, $G= {\rm Gal}(M/K)$ be its (arithmetic) Galois group and $H={\rm Hom}_{F}(L,M)$ be the set of field $F$-homomorphisms from $L$ to $M$. 
Let $\mathcal P$ be a place of $F$, $Q$ be the set of places of $L$ lying above $\mathcal P$, and $\mathcal R$ be a place of $M$ lying above $\mathcal P$.

Then there is a natural bijection $\beta$ between $Q$ and the set of orbits of $H$ under the natural action of the decomposition group $D(\mathcal R|\mathcal P)$.
For any $\mathcal Q \in Q$, the orbit $\beta(\mathcal Q)$ has size
\[|\beta(\mathcal Q)|=e(\mathcal Q|\mathcal P)\cdot f(\mathcal Q|\mathcal P),\]
where $e(\mathcal Q|\mathcal P)$ and $f(\mathcal Q|\mathcal P)$ are respectively the ramification index and the relative degree of $\mathcal Q$ over $\mathcal P$.
\end{lemma}

\subsection{Linear sets}

%Let $\Lambda=\PG(V,\F_{q^n})=\PG(r-1,q^n)$, where $V$ is a vector space of dimension $r$ over $\F_{q^n}$.
%A point set $L$ of $\Lambda$ is said to be an \emph{$\F_q$-linear set} of $\Lambda$ of rank $k$ if it is
%defined by the non-zero vectors of a $k$-dimensional $\F_q$-vector subspace $U$ of $W$, i.e.
%\[L=L_U=\{\la {\bf u} \ra_{\mathbb{F}_{q^n}} \colon {\bf u}\in U\setminus \{{\bf 0} \}\}.\]
%Also,the \emph{weight of a point} $P=\langle \mathbf{u} \rangle_{\F_{q^n}}$ is $w_{L_U}(P)=\dim_{\F_q}(U\cap\langle \mathbf{u} \rangle_{\F_{q^n}})$.

Let $V$ be an $r$-dimensional $\fqn$-vector space and $L=L_U$ be an $\fq$-linear set of rank $h$ in $\Lambda=\PG(V,\fqn)$.
The \emph{maximum field of linearity} of $L_U$ is the largest subfield $\F_{q^m}$ of $\fqn$ such that $L_U$ is $\F_{q^m}$-linear.
If $N_i$ denotes the number of points of $\Lambda$ having weight $i\in \{0,\ldots,h\}$ in $L_U$, the following relations hold:
\begin{comment}
\begin{equation}\label{eq:card}
    |L_U| \leq \frac{q^k-1}{q-1},
\end{equation}
\begin{equation}\label{eq:pesicard}
    |L_U| =N_1+\ldots+N_k,
\end{equation}
\begin{equation}\label{eq:pesivett}
    N_1+N_2(q+1)+\ldots+N_k(q^{k-1}+\ldots+q+1)=q^{k-1}+\ldots+q+1.
\end{equation}
\end{comment}
\[
|L_U|\leq \frac{q^h-1}{q-1},\qquad
\sum_{i=0}^{h}N_i\cdot \frac{q^i-1}{q-1} = \frac{q^h-1}{q-1}.
\]
Furthermore, $L_U$ is called \emph{scattered} if it has the maximum number $\frac{q^h-1}{q-1}$ of points, or equivalently, if all points of $L_U$ have weight one.
%We say that two linear sets $L_U$ and $L_W$ of $\Lambda=\PG(r-1,q^n)$ are $\mathrm{P}\Gamma \mathrm{L}$-\emph{equivalent} if there exists $\varphi \in \mathrm{P}\Gamma \mathrm{L} (r-1,q^n)$ such that $\varphi(L_U)=L_W$.

If two $\fq$-subspaces $U_1$ and $U_2$ of $V$ belong to the same orbit of $\Gamma\mathrm{L}(r,q^n)$, then $L_{U_1}$ and $L_{U_2}$ are clearly $\mathrm{P}\Gamma \mathrm{L}$-equivalent.
This sufficient condition is not necessary in general: there exist $\fq$-linear sets $L_{U_1}$ and $L_{U_2}$ in the same orbit of ${\rm P\Gamma L}(r,q^n)$ such that $U_1$ and $U_2$ are not in the same orbit of ${\rm \Gamma L}(r,q^n)$. %; see \cite{CsMP}.
The $\fq$-linear set $L_U$ is called \emph{simple} if every $\fq$-subspace $W$ of $V$ satisfying $L_U=L_W$ is in the same $\Gamma\mathrm{L}(r,q^n)$-orbit as $U$; 
examples of simple linear sets are given by the subgeometries of $\PG(V,\fqn)$. 
For further details on the equivalence issue for linear sets we refer to \cite{CsMP}.

\smallskip

The linear sets considered in this work are $\fq$-linear sets $L$ of rank $n$ in $\PG(1,q^n)$.
Since ${\rm P\Gamma L}(1,q^n)$ is $3$-transitive on $\PG(1,q^n)$, we can assume up to equivalence that $L$ does not contain the point $\langle(0,1)\rangle_{\fqn}$.
As already noticed in the introduction, this implies for any non-negative integer $t$ that $L$ equals $L_{f,t}$ as defined in \eqref{eq:defLft} for some $q$-polynomial $f(x)\in\fqn[x]$.
Without restriction, the index $t$ will be always be assumed to be in $\{0,\ldots,n-1\}$.
Also, when $t=0$, we denote $L_{f,0}$ simply by $L_f$.

An $\fq$-linear set $L_{f,t}$ is said to be of \emph{pseudoregulus type} if it is ${\rm P\Gamma L}$-equivalent to $L_{x^{q^s}}$, where $s$ satisfies $\gcd(s,n)=1$, see \cite{LMPT:14}.
In this case, $L_{f,t}$ is scattered.

Consider the non-degenerate symmetric bilinear form of $\F_{q^n}$ over $\F_q$ defined for every $x,y \in \F_{q^n}$ by
\begin{equation}\label{eq:bilform} \la x,y\ra= \mathrm{Tr}_{q^n/q}(xy). \end{equation}
\noindent The \emph{adjoint} $\hat{f}$ of the $q$-polynomial $\displaystyle f(x)=\sum_{i=0}^{n-1} a_ix^{q^i} \in \F_{q^n}[x]$ with respect to the bilinear form $\la\cdot,\cdot\ra$ is
\[ \hat{f}(x)=\sum_{i=0}^{n-1} a_i^{q^{n-i}}x^{q^{n-i}}, \]
and is the unique function over $\fqn$ satisfying
\[ \mathrm{Tr}_{q^n/q}(yf(z))=\mathrm{Tr}_{q^n/q}(z\hat{f}(y)), \]
for every $y,z \in \F_{q^n}$.

\begin{proposition}\label{prop:adjoint}{\rm (\cite[Lemma 2.6]{BGMP},\cite[Lemma 3.1]{CsMP})}
Let $f(x)$ be a $q$-polynomial over $\fqn$ and $\hat{f}(x)$ be its adjoint w.r.t. the bilinear form \eqref{eq:bilform}. Then $L_f=L_{\hat{f}}$.
\end{proposition}

\begin{remark}\label{rk:weight}
Let $f(x)$ be a $q$-polynomial over $\fqn$.
For any $m\in\fqn$, the weight on the point $\langle(1,m)\rangle_{\fqn}\in\PG(1,\fqn)$ in $L_{f,t}$ satisfies
\[ w_{L_{f,t}}(\langle (1,m) \rangle_{\fqn})= \dim_{\fq} \ker(f(x)-mx^{q^t}). \]
Clearly, $w_{L_{f,t}}(\langle(0,1)\rangle_{\fqn})=0$.
\end{remark}

%An $i$-\emph{club} of rank $n$ in $\PG(1,q^n)$ is an $\F_{q}$-linear set in $\PG(1,q^n)$ such that one point has weight $i$ and all the others have weight one.
%By \eqref{eq:pesicard} and \eqref{eq:pesivett}, we have that an $i$-club has cardinality $q^{n-1}+\ldots+q^i+1$.

For further details on linear sets see \cite{LavVdV,Polverino}.

\section{On the weight distribution of \texorpdfstring{$r$}{r}-fat polynomials}\label{sec:difficile}

Through this section, $f(x)=\sum_{j=0}^k a_j x^{q^j}\in\fqn[x]$ is an $r$-fat polynomial of index $t$ and $q$-degree $k<n$.
For any $i\geq1$, $N_i$ denotes the number of points of weight $i$ in $L_{f,t}$.

Inspired by \cite{BZ,BM,FM2020}, we provide bounds on the weight distribution of $f(x)$ and $L_{f,t}$, i.e. on $r$ and the $N_i$'s, involving $q,n,k,t$.
In particular, we prove non-existence results for exceptional $r$-fat polynomials, which generalize non-existence results for exceptional scattered polynomials proved in \cite{BZ,BM,FM2020}.

\begin{remark}\label{rem:assumptions1}
If $f(x)=a_k x^{q^{k}}$, then an easy change of variable shows $L_{f,t}=L_{a_k x^{q^{k-t}},0}$.
Hence, the maximum field of linearity of $L_{f,t}$ is $\mathbb{F}_{\bar q}$ with $\bar{q}=q^{\gcd(n,k-t)}$ and $L_{f,t}$ is a scattered $\mathbb{F}_{\bar q}$-linear set of pseudoregulus type.
Thus, the number of points of $L_{f,t}$ is $\frac{q^n-1}{\bar{q}-1}$, all of them having weight $\gcd(n,k-t)$.
We will then assume without restriction that $f(x)$ is not a monomial.
\end{remark}

\begin{remark}\label{rem:assumptions2}
In the literature, a polynomial $f(x)$ is called \emph{$t$-normalized} if the following conditions hold (see \cite[p. 511]{BZ}):
\begin{itemize}
    \item $f(x)$ is monic;
    \item the coefficient $a_t$ of $x^{q^t}$ in $f(x)$ is zero;
    \item if $t>0$, then the coefficient $a_0$ of $x$ in $f(x)$ is nonzero, i.e. $f(x)$ is separable.
\end{itemize}
For any index $t$, if $|L_{f,t}|>1$ then the weight distribution of $L_{f,t}$ is equivalent to the weight distribution of $L_{f^{\prime},t^{\prime}}$ for some $t^\prime$-normalized polynomial $f^\prime(x)\in\fqn[x]$ and index $t^\prime$.

In fact, the weight of $\langle(1,m)\rangle_{\fqn}$ in $L_{f,t}$ is equal to the weight of $\langle(1,m/a_k)\rangle_{\fqn}$ in $L_{\frac{1}{a_k}f(x),t}$, and to the weight of $\langle(1,m-a_t)\rangle_{\fqn}$ in $L_{f(x)-a_tx^{q^t},t}$; also, $L_{f,t}=L_{g,t^\prime}$, where $t^\prime=t-\min\{t,\min\{j\colon a_j\ne0\}\}$ and $g(x)=f(x^{q^{t^\prime}})$.
%Therefore, the investigation of weight distributions of $r$-fat polynomials can be restricted to $t$-normalized polynomials, for $t\geq0$.
\end{remark}

\begin{theorem}\label{Th:Curve}
Let $t\geq0$ and $f(x)=\sum_{j=0}^{k}a_j x^{q^j}\in\fqn[x]$ be a monic $\fq$-linearized polynomial of $q$-degree $k<n$ such that $a_t=0$ and $f(x)$ is not a monomial.
Let $v=\min\{j\colon a_j\ne0\}$, and assume that $v=0$ whenever $t>0$.
Let $h=\dim_{\fq}\ker(f)$.
For any $i\geq 1$, let $N_i$ be the number of points having weight $i$ in $L_{f,t}$.

If $t=0$, then
\[
\sum_{i=1}^{n-1}N_i(q^i-1)(q^i-q)\geq q^n - (q^k-q-1)(q^k-q-2)\sqrt{q^n} -2q^{k-v}-q^k+q^{h+1}-q^h+1.
\]
If one of the following two conditions hold:
\begin{itemize}
    \item[(i)] $t=1$ and $k\geq 3$;
    \item[(ii)] $t\geq2$, with $k \geq 3t$ if  $t \mid k$ or $k \geq 2t-1$ if $t \nmid k$, and  either
    \begin{itemize}
        \item $f(x)=a_0 x+a_1 x^q+\sum_{j>t}a_j x^{q^j}$ with $k \geq t+2$, or 
        \item $f(x)=a_0 x+\sum_{j>t}a_j x^{q^j}$;
\end{itemize}
\end{itemize}
then
\[
\sum_{i=1}^{n-1}N_i(q^i-1)(q^i-q)\geq q^n - (q^k+q^t-q-2)(q^k+q^t-q-3)\sqrt{q^n} -q^k-q^t+q+1.
\]
\end{theorem}

\begin{proof}
Recall from \cite{BZ} that two nonzero vectors $(\bar{x}^{q^t},f(\bar{x})),(\bar y^{q^t},f(\bar y))\in\mathbb{F}_{q^n}^2$ define the same point $P$ of $L_{f,t}$ if and only $(\bar x,\bar y)$ is an $\mathbb{F}_{q^n}$-rational affine point of the reducible plane curve $\mathcal{N}$ with affine equation
\[
\mathcal{N}\colon\quad f(X)Y^{q^t}-f(Y)X^{q^t}=0.
\]
Hence, each point $P$ of weight $w_{L_{f,t}}(P)=i\geq1$ in $L_{f,t}$ corresponds to exactly $(q^i-1)^2$ $\fqn$-rational affine points of $\mathcal{N}$, namely to the points $(\bar{x},\bar{y})$ where $\bar x$ and $\bar y$ are nonzero elements of the $i$-dimensional $\fq$-subspace $W_P$ of $\fqn$ defining $P$.

Let $\mathcal{C}$ be the plane curve with affine equation
\begin{equation}\label{Eq:Curva}
\mathcal{C}\colon\quad \frac{f(X)Y^{q^t}-f(Y)X^{q^t}}{X^q Y- X Y^q}=0.
\end{equation}

Suppose that $t=0$.
By direct computation, the affine point $(\bar x,\bar y)\ne(0,0)$ is singular for $\mathcal{N}$ if and only if $f(\bar x)=f(\bar y)=0$.

If $P\ne\langle(1,0)\rangle_{\fqn}\in L_{f,t}$ and $(\bar{x},\bar{y})\in\mathcal{N}$ corresponds to $P$, then $(\bar x,\bar y)$ is not singular for $\mathcal{N}$; hence, $(\bar x,\bar y)$ is a point of $\mathcal{C}$ if and only $\bar x,\bar y\in W_P^*$ and $\bar y/\bar x\notin \fq$. Thus, $P$ corresponds exactly to $(q^i-1)(q^i-q)$ points of $\mathcal{C}$.

If $P=\langle(1,0)\rangle_{\fqn}\in L_{f,t}$ and  $(\bar{x},\bar{y})\in\mathcal{N}$ corresponds to $P$, then $(\bar x,\bar y)$ is singular for $\mathcal{N}$ and hence $(\bar x,\bar y)$ is a point of $\mathcal{C}$ for any $\bar x,\bar y\in W_P^*$. Thus, $P$ corresponds exactly to $(q^h-1)^2$ points of $\mathcal{C}$.

The possible $\fqn$-rational affine points of $\mathcal{C}$ not corresponding to any point of $L_{f,t}$ are the following ones.
\begin{itemize}
    \item The origin $(0,0)$, which is a point of $\mathcal{C}$ if and only if $s>1$.
    \item The $\fqn$-rational affine points $(\bar x,\bar y)$ on the axes $X=0$ or $Y=0$ different from the origin and satisfying $f(\bar x)=f(\bar y)$, whose number is at most $2(q^{k-v}-1)$.
\end{itemize}
The $\fqn$-rational points at infinity of $\mathcal{C}$ are at most $\deg(\mathcal{C})=q^k-q$.

Summing up, the number of $\fqn$-rational points of $\mathcal{C}$ is at most
\begin{equation}\label{eq:quanti}
\sum_{i\geq1,\,i\ne h}N_i(q^i-1)(q^i-q)+(N_h-1)(q^h-1)(q^h-q)+(q^h-1)^2 +1+2(q^{k-s}-1)+q^k-q.
\end{equation}
By the results in \cite{BZ,BM}, $\mathcal{C}$ contains an absolutely irreducible component $\mathcal{X}$ defined over $\fqn$, whose degree is at most the degree of $\mathcal{C}$; see \cite[Theorems 3.1 and 3.5]{BZ}, \cite[Proposition 3.6 and Section 4]{BM}.
Therefore, the number of $\fqn$-rational points of $\mathcal{X}$ is upper bounded by the value in \eqref{eq:quanti}, and lower bounded by
$q^n+1-(q^k-q-1)(q^k-q-2)\sqrt{q^n}$
because of Theorem \ref{th:hasse-weil}.
This implies
\[
\sum_{i\geq1,\,i\ne h}N_i(q^i-1)(q^i-q)+(N_h-1)(q^h-1)(q^h-q)+(q^h-1)^2 +1+2(q^{k-v}-1)+q^k-q
\]
\[
\geq q^n+1-(q^k-q-1)(q^k-q-2)\sqrt{q^n},
\]
whence the claim for $t=0$ follows.

For $t>0$, the same arguments can be used. The main difference is that the origin is a $(q^t-q)$-fold point of $\mathcal{C}$, while no other affine point of $\mathcal{N}$ is singular.
Thus, each point of weight $i\geq1$ in $L_{f,t}$ corresponds to exactly $(q^i-1)(q^i-q)$ points of $\mathcal{C}$.
Also, $\mathcal{C}$ has no affine points on the axes $X=0$ or $Y=0$ other than the origin.
By the Hasse-Weil lower bound applied to an $\fqn$-rational absolutely irreducible component $\mathcal{X}$ of $\mathcal{C}$, the claim follows.
\end{proof}

\begin{corollary}\label{cor:rcurve}
%Let $t\geq0$ and $f(x)=\sum_{j=0}^{k}a_j x^{q^j}\in\fqn[x]$ be a monic $\fq$-linearized polynomial of $q$-degree $k<n$ such that  $a_t=0$ and $f(x)$ is not a monomial.
%Let $v=\min\{j\colon a_j\ne0\}$, and assume that $v=0$ whenever $t>0$.
%Define $h=\dim_{\fq}\ker(f)$.
If $f(x)$ is an $r$-fat polynomial of index $t$ and maximum weight $W\geq 2$ over $\fqn$ satisfying the assumptions of Theorem \ref{Th:Curve}, then
\begin{equation}\label{eq:rbound}
r\geq \frac{q^n-c_1\sqrt{q^n}-c_2}{q^{2W}-q^{W+1}-q^W+q},
\end{equation}
where
\[
c_1=(q^k+q^t-q-2)(q^k+q^t-q-3),\quad
c_2=\begin{cases}
2q^{k-v}+q^k-q^{h+1}+q^h-1 & \mbox{if  } t=0, \\
q^k+q^t-q-1 & \mbox{if  } t>0. \\
\end{cases}
\]
In particular:
\begin{itemize}
\item $f(x)$ is not exceptional $r$-fat of index $t$;
\item if one of the following holds:
\begin{itemize}
    \item $q\geq4$ and $\max\{k,t\}<\frac{n}{4}$, or
    \item $q=3$ and $\max\{k,t\}\leq\frac{n-2}{4}$, or
    \item $q=2$ and $\max\{k,t\}\leq \frac{n}{4}-1$,
\end{itemize}
then $r>1$, i.e. $L_{f,t}$ is not a $W$-club.
\end{itemize}
\end{corollary}

\begin{proof}
The bound \eqref{eq:rbound} follows from Theorem \ref{Th:Curve} and
\[
 \sum_{i=1}^{n-1}N_i(q^i-1)(q^i-q)\leq r(q^W-1)(q^W-q).
\]
Therefore $f(x)$ is not exceptional $r$-fat because the bound
\[\frac{q^{n}-c_1\sqrt{q^{n}}-c_2}{q^{2W}-q^{W+1}-q^W+q}\leq r\]
is not satisfied whenever $m$ is big enough.

If $q\geq 4$ and $\max\{k,t\}<n/4$, or $q=3$ and $\max\{k,t\}\leq(n-2)/4$, or $q=2$ and $\max\{k,t\}\leq \frac{n}{4}-1$, then direct computation shows $q^n-c_1\sqrt{q^n}-c_2>0$.
Using $2\leq W\leq\max\{k,t\}$, this yields 
\[ 0<q^{2W}-q^{W+1}-q^W+q< q^n-c_1\sqrt{q^n}-c_2. \]
Together with \eqref{eq:rbound}, this implies $r>1$.
\end{proof}

\begin{theorem}\label{Th:Chebocazzi}
Let $t\geq1$ and $f(x)=\sum_{j=0}^{k}a_j x^{q^j}\in\fqn[x]$ be an $\fq$-linearized polynomial of $q$-degree $k<n$ such that $a_t=0$ and $a_0\ne0$.
Let $d=\max\{k,t\}$.
Let $N_i$ be the number of points of weight $i$ in $L_{f,t}$.

For any $i$, we have either $N_i=0$ or
\begin{equation}\label{eq:claimbbello}
N_i>\frac{q^n+1}{q^{d^2}}-\frac{2(d-1)q^{d^2}}{q^d-1}\sqrt{q^n}-1.
\end{equation}
\end{theorem}

\begin{proof}
Define the polynomial $g(s,x)=f(x)/x-sx^{q^t-1}\in\fqn[s,x]$, which is irreducible over $\overline{\mathbb{F}}_{q^n}$, and consider the extension $L/F$ of global function fields over $\fqn$, where $F=\fqn(s)$ with $s$ transcendental over $\fqn$, and $L=\fqn(s,z)$ with $g(s,z)=0$.
Let $s_0\in\fqn$ be such that
\[
i:=\dim_{\fq}\ker(f(x)-s_0x^{q^t})\geq1,
\]
so that $L$ has $q^i$ places of degree $1$ lying over the place $\mathcal{P}_{s_0}$ in $F$. Then $N_i$ is at least the number $A$ of places at finite of degree $1$ of $F$ having the same factorization as $\mathcal{P}_{s_0}$ under $L$.

Let $M/\fqn$ be the Galois closure of $L/F$; clearly, $M$ is the splitting field of $g(s,x)\in F[x]$.
Note that $D(\mathcal R|\mathcal P)$ acts on  $H={\rm Hom}_F(L,M)$ as on the roots of $g(s,x)\in F[x]$, because any $\tau\in H$ is uniquely determined by the root $\tau(z)$ of $g(s,x)$. 
Since $g(s,x)$ has nonzero constant derivative, then it is squarefree for any specialization $s_0\in\overline{\mathbb F}_{q^n}$ of $s$, so the extension $L/F$ is unramified at any place at finite, and therefore $M/F$ is also unramified at any place at finite (see \cite[Corollary 3.9.3]{Sti}).
Let $g_{\mathcal{P}_{s_0}}$ be the Frobenius for $\mathcal{P}_{s_0}$: in fact, since $\mathcal{P}_{s_0}$ is unramified we have canonical isomorphisms  
$D(\mathcal R|\mathcal{P}_{s_0})\cong{\rm Gal}(\mathcal{O}_{\mathcal R}/\mathcal R\colon \mathcal{O}_{\mathcal{P}_{s_0}}/\mathcal{P}_{s_0})\cong
{\rm Gal}(\mathbb{F}_{q^{n\cdot f(\mathcal R|\mathcal P)}}\colon\fqn)$, and therefore one can choose the preimage $g_{\mathcal{P}_{s_0}}\in D(\mathcal R|\mathcal{P}_{s_0})$ of the map $(u\mapsto u^{q^n})\in {\rm Gal}(\mathbb{F}_{q^{n\cdot f(\mathcal R|\mathcal P)}}\colon\fqn)$.

By Lemma \ref{orbits}, the cycle decomposition of $g_{\mathcal{P}_{s_0}}$ on $H$ identifies the splitting of $\mathcal P_{s_0}$ under $L$; in particular, the number of cycles of length one of $g_{\mathcal P_{s_0}}$ on $H$ (i.e. the fixed points of $g_{\mathcal P_{s_0}}$) equals the number of places of degree $1$ of $L$ lying over $\mathcal{P}_{s_0}$.

Thus, the claim follows once that the right-hand side in \eqref{eq:claimbbello} is smaller than the number $U$ of places $\mathcal{P}$ at finite of degree $1$ of $F$ such that $g_{\mathcal{P}_{s_0}}$ is a Frobenius at $\mathcal{P}$, since $U\leq A\leq N_i$.

We apply Theorem \ref{th:chebotarev} and  we obtain 
 \[ U\geq \frac{|\Gamma_{\sigma}|}{|G^{\rm geom}|}(q^n+1)-\frac{2g_M\cdot|\Gamma_{\sigma}|}{|G^{\rm geom}|}\sqrt{q^n} - 1;\]
here we use that there is at most one place of degree $1$ of $F$ which ramifies under $M$, namely the place $\mathcal{P}_{\infty}$ at infinity, and $|\Gamma_{\sigma}|\cdot w_{\mathcal{P}_{\infty}}(\sigma)\leq1$.

In order to give an upper bound for $g_M$, let $\{x_1,\ldots,x_d\}$ be an $\fq$-basis of $V$, so that $M=\fqn(s,x_1,\ldots,x_d)$ is the compositum of the $d$ rational function fields
\[
F_1=\fqn(s,x_1)=\fqn(x_1),\ldots,F_d=\fqn(s,x_d)=\fqn(x_d).
\]
For every $j=1,\ldots,d-1$, the compositum $F_1\cdots F_j=\fqn(x_1,\ldots,x_j)$ contains at least $q^j$ roots of $f_s(x)$, so that \[[F_1\cdots F_{j+1}\colon F_1\cdots F_{j}]\leq (q^d-q^j)<q^d,\]
\[
[F_1\cdots F_{j+1}\colon F_{j+1}]=
\prod_{i=1}^j [F_{j+1} F_1\cdots F_i\colon F_{j+1}F_1\cdots F_{i-1}]\leq \prod_{i=1}^j (q^d-q^i)<q^{jd}.
\]
Thus, the inductive application of Castelnuovo's Inequality (see \cite[Theorem 3.11.3]{Sti}) provides the upper bound
%\[ g(F_1F_2)< q^d\cdot0+q^d\cdot0+(q^d-1)^2<q^{2d},\]
%\[ g(F_1\cdots F_{j+1})< q^d\cdot g(F_1\cdots F_{j})+q^{jd}\cdot 0+ (q^d-1)(q^{jd}-1) < j\cdot q^{(j+1)d}.\]
\[ g_{F_1\cdots F_{j+1}}<j q^{(j+1)d} \]
for the genus of $F_1\cdots F_{j+1}$. In particular, $g_M< (d-1)q^{d^2}$.
Together with $|\Gamma_{\sigma}|\geq1$, this yields
\begin{equation}\label{eq:U}
U> \frac{q^n+1}{|G^{\textrm{geom}}|}-\frac{2(d-1)q^{d^2}}{|G^{\textrm{geom}}|}\sqrt{q^n}-1.
\end{equation}
%We now want to bound $|G^{\rm geo}|$ from above and below, so that we can give a bound on $U$ independent of $|G^{\rm geo}|$.
In order to give an upper bound on $|G^{\rm geom}|$, note that $G^{\rm arith}={\rm Gal}(M/F)$ embeds into ${\rm GL}(d,q)$, because the roots of $x\cdot g(s,x)$ form a $d$-dimensional $\fq$-vector space $V$ and $G^{\rm arith}$ has an $\fq$-linear faithful action on $V$. Therefore
\[|G^{\rm geom}|\leq |G^{\rm arith}|\leq |{\rm GL}(d,q)|=\prod_{j=1}^{d-1}(q^d-q^j)<q^{d^2}.\]
%On the other hand, since $F/x$ is irreducible over $\overline{\vF}_{q^n}(t)$, it follows that the roots of $F/x$ in $\overline {\vF_{q^n}(t)}$ are $q^d-1$. From this follows that the action of $G^{\rm geom}$ is transitive and therefore $|G^{\rm geom}|\geq q^d-1$.
By plugging this bound in \eqref{eq:U} and recalling $N_i\geq U$, the claim follows.
\end{proof}

\begin{corollary}\label{rchebo}
Let $t\geq1$ and $f(x)=\sum_{j=0}^{k}a_j x^{q^j}\in\fqn[x]$ be an $\fq$-linearized polynomial of $q$-degree $k<n$ such that $a_t=0$ and $a_0\ne0$.
Let $d=\max\{k,t\}$.

If $f(x)$ is $r$-fat of index $t$ with $r>0$, then
\[
r>\frac{q^n+1}{q^{d^2}}-\frac{2(d-1)q^{d^2}}{q^d-1}\sqrt{q^n}-1.
\]
In particular, $f(x)$ is not exceptional $r$-fat of index $t$.
\end{corollary}

%\begin{corollary}\label{cor:rchebo}
%Let $t\geq1$ and $f(x)=\sum_{j=0}^{k}a_j x^{q^j}\in\fqn[x]$ be an $\fq$-linearized polynomial of $q$-degree $k<n$ such that $a_t=0$, $a_0\ne0$ and $f(x)$ is not a monomial.
%Let $d=\max\{k,t\}$.

%If $r>0$, then $f(x)$ is not exceptional $r$-fat of index $t$.
%\end{corollary}

\begin{remark}\label{rem:eppoi}
Suppose that $t=0$, and assume without restrictions that $f(x)=\sum_{j=0}^{k}a_j x^{q^j}\in\fqn[x]$ is not a monomial, with $k<n$ and $a_0=0$.
If $s_0\in\fqn$, then the polynomial $f(x)-s_0x^{q^t}=f(x)-s_0x\in\fqn[x]$ is separable if and only if $s_0\ne0$.
Therefore the argument in the proof of Theorem \ref{Th:Chebocazzi} still yields Condition \eqref{eq:claimbbello} for $N_i$ whenever $L_{f,0}$ has a point of weight $i$ different from $\langle(1,0)\rangle_{\fqn}$.
Neverthless, in the case $t=0$, a better bound is provided by Theorem \ref{Th:Curve}.
\end{remark}

%Corollaries \ref{cor:rcurve} and \ref{cor:rchebo} imply  the following restriction on exceptional $r$-fat polynomials.

\begin{theorem}\label{th:solo1}
There exist no exceptional $r$-fat polynomials whenever $r>0$.
\end{theorem}

\begin{proof}
If an exceptional $r$-fat polynomial with $r>0$ exists, then by Remarks \ref{rem:assumptions1} and \ref{rem:assumptions2} there exists a $t$-normalized exceptional $r$-fat polynomial $f(x)\in\fqn[x]$ for some index $t$, which is not a monomial.
This is a contradiction to Corollary \ref{cor:rcurve} in the case $t=0$, and to Corollary \ref{rchebo} in the case $t>0$.  
\end{proof}

%\begin{remark}\label{RemarkSupercazzoloso}
%By Theorem \ref{cor:solo1}, any exceptional $r$-fat polynomial over $\fqn$ is equivalent to a $1$-fat polynomial $f(x)\in\fqn[x]$ of index zero and maximum weight $\dim_{\fq}\ker(f)$.
%A possible strategy to find examples of $r$-fat polynomials (which cannot be exceptional by Theorem \ref{th:solo1}) relies on the investigation of the algebraic curve $\mathcal{C}$ in \eqref{Eq:Curva}. One may  search for linearized polynomials $f(x)$ such that all absolutely irreducible components of $\mathcal{C}$ are not defined over $\mathbb{F}_{q^n}$ and intersect in a ``small" number of $\mathbb{F}_{q^n}$-rational points.
%To be more precise, and using the same notation as in Theorem \ref{Th:Curve}, first observe that in this case the curve $\mathcal{C}$ possesses at most $(q^k+q^t-q-1)^2/4$ $\mathbb{F}_{q^n}$-rational points  (see  for instance \cite[Lemma 3.3]{VV}). 
%This means that $\sum_{i}N_i (q^i-1)(q^i-q)\leq (q^k+q^t-q-1)^2/4$.
%This means that $(q^W-1)(q^W-q)\leq (q^k+q^t-q-1)^2/4$.
%The same bound holds over the infinitely many extensions $\mathbb{F}_{q^{nm}}$ over which no absolutely irreducible component of the curve $\mathcal{C}$ is defined.
%Thus, there could be room  for $1$-fat polynomials of index zero over $\mathbb{F}_{q^{nm}}$.
%\end{remark}

\section{Known examples of \texorpdfstring{$1$}{1}-fat polynomials}\label{sec:known}

Let $f(x)$ be an $\fq$-linearized polynomial over $\fqn$. Recall that $f(x)$ is $1$-fat of index $t$ and maximum weight $i$ if and only if $L_{f,t}$ is an $i$-club of rank $n$ in $\PG(1,q^n)$.

Let $q$ be any prime power, $n,\ell,m,s,i \in \mathbb{N}$ be such that $n=\ell m$, $i=m(\ell-1)$, $\gcd(s,n)=1$.
Consider over $\fqn$ the automorphism $\sigma:x\mapsto x^{q^s}$ and the trace function $\mathrm{Tr}_{q^{\ell m}/q^m}:x\mapsto \sum_{j=0}^{\ell-1}x^{q^{jm}}$.
Define $T(x)=\mathrm{Tr}_{q^{\ell m}/q^m}\circ \sigma(x)\in\fqn[x]$.
In \cite[Theorem 3.3]{DeBoeckVdV}, the authors proved that $T(x)$ is a $1$-fat polynomial of index zero and maximum weight $i$ over $\fqn$.

It is known that every $1$-fat polynomial of maximum weight $n-1$ over $\fqn$ is equivalent to the trace function $\mathrm{Tr}_{q^n/q}$ (see \cite[Theorem 2.3]{DeBoeckVdV}).
Clearly, $L_{\mathrm{Tr}_{q^n/q},t}=L_{\mathrm{Tr}_{q^n/q},0}$ for any index $t$.
Hence, $\mathrm{Tr}_{q^n/q}(x)$ is a $1$-fat polynomial of maximum weight $n-1$ over $\fqn$ for any index, and any $1$-fat polynomial of maximum weight $n-1$ over $\fqn$ is equivalent to $\mathrm{Tr}_{q^n/q}(x)$ of index $0$.
As a consequence, we have a complete classification of $1$-fat polynomials over $\fqn$ for the case $n=3$.

We resume the aforementioned results as follows.
\begin{theorem}\label{th:resume}
\begin{itemize}
    \item Any $1$-fat polynomial over $\fqn$ with maximum weight $n-1$ is equivalent to the trace function $\mathrm{Tr}_{q^n/q}(x)$.
    \item Let $n=\ell m$, $i=m(\ell-1)$, $\gcd(s,n)=1$ and $\sigma:x\mapsto x^{q^s}$. The polynomial $T(x)=\mathrm{Tr}_{q^{\ell m}/q^m}\circ \sigma(x)$ is a $1$-fat polynomial of index $mj$ over $\F_{q^{\ell m}}$ with maximum weight $i$, for every $j\in\{0,\ldots,\ell-1\}$.
    \item Any $1$-fat polynomial over $\F_{q^3}$ is equivalent to the trace function $\mathrm{Tr}_{q^3/q}(x)$.
\end{itemize}
\end{theorem}

Further examples of clubs of rank $n$ in $\PG(1,q^n)$ are not described in terms of linearized polynomials; it is not known for which linearized polynomial $f(x)$ and index $t$ they are equal to $L_{f,t}$.
For instance, in \cite[Lemma 2.12]{DeBoeckVdV} it was proved that for any $\lambda \in \fqn^*$ such that $\{1,\lambda,\ldots,\lambda^{n-1}\}$ is an $\fq$-basis of $\fqn$, the $\fq$-linear set
\begin{equation}\label{eq:exKMnopol} L_{\lambda}=\{ \langle (t_1\lambda+\ldots+t_{n-1}\lambda^{n-1},t_{n-1}+t_n \lambda) \rangle_{\fqn} \colon (t_1,\ldots,t_n)\in \fq^n\setminus\{\mathbf{0}\} \},
\end{equation}
is an $(n-2)$-club of rank $n$ in $\PG(1,q^n)$.
Other examples have been described in \cite[Theorems 3.6 and 3.12, Corollary 4.15]{DeBoeckVdV}.

Clubs of rank $n$ in $\PG(1,q^n)$, and hence $1$-fat polynomials over $\fqn$, are connected with interesting nonlinear combinatorial objects from finite geometry, namely with so-called KM-arcs, as shown by De Boeck and Van de Voorde in \cite[Theorems 2.1 and 2.2]{DeBoeckVdV}. 
A \emph{KM-arc $\mathcal{A}$ of type} $\tau$ in $\PG(2,q)$ is a set of $q+\tau$ points of type $(0,2,\tau)$, i.e. each line of $\PG(2,q)$ meets $\mathcal{A}$ in either 0, 2 or $\tau$ points.
Korchm\'aros and Mazzocca in \cite{KM} proved that, if a KM-arc of type $\tau$  in $\PG(2,q)$ exists and $2<\tau<q$, then $q$ is even and $\tau$ is a divisor of $q$.
A KM-arc $\mathcal{A}$ in $\PG(2,q)$ is called a \emph{translation} KM-arc if there exists a line $\ell$ of $\PG(2,q)$ such that the group of elations with axis $\ell$ and fixing $\mathcal{A}$ acts transitively on the points of $\mathcal{A}\setminus \ell$. 
The first example of a KM-arc of type $2^i$ was presented in \cite{KM} and can be described as $L_{T}$, where $T(x)=\mathrm{Tr}_{q^{\ell m}/q^m}\circ\sigma (x)$ is described in Theorem \ref{th:resume} with $q=2$; see \cite[Theorem 3.2]{DeBoeckVdV}.

\section{\texorpdfstring{$r$}{Lg}-fat polynomials over \texorpdfstring{$\F_{q^4}$}{Lg} and linear sets in \texorpdfstring{$\PG(1,q^4)$}{Lg}}\label{sec:q4}

In this section we use the results of \cite{BoPol} and \cite{LavVdV2010} to give a complete classification of $\fq$-linear sets of rank $4$ in $\PG(1,q^4)$. 
%yielding \textcolor{blue}{(NO)} to the classification of $1$-fat polynomials over $\mathbb{F}_{q^4}$.

First we recall some definition on blocking sets.
A \emph{blocking set} $\mathcal{B}$ of $\PG(2,q^n)=\PG(\overline{V},\fqn)$ is a point set intersecting every line of $\PG(2,q^n)$.
Furthermore, if $\mathcal{B}$ has size $q^n+N\leq 2q^n$ and there exists a line $\ell$ meeting $\mathcal{B}$ in exactly $N$ points, then $\mathcal{B}$ is called of \emph{R\'edei type} and $\ell$ is said to be a \emph{R\'edei line}.

A way to construct blocking sets of R\'edei type is the following.
Let $L_U$ be an $\fq$-linear set of rank $n$ of a line $\ell=\PG(V,\fqn)$, where $V$ is a $2$-dimensional $\fqn$-subspace of $\overline{V}$, and let $\mathbf{v}\in \overline{V}\setminus V$.
Then the linear set $L_{\langle U, \mathbf{v}\rangle_{\fq}}$ of $\PG(2,q^n)$ is an $\fq$-linear blocking set of R\'edei type with R\'edei line $\ell$.

\begin{remark}
Note that both the $\fq$-linear sets of rank $n+1$ (that is the $\fq$-linear blocking sets)
%\textcolor{blue}{(ok that rank $n+1$ implies blocking set, but conversely: what if the linear blocking set has rank bigger than $n+1$? Is it trivial, or non-minimal? Please explain to me your 'that is' ;-) )} 
in $\PG(2,q^n)$ and the $\fq$-linear sets of rank $4$ in $\PG(1,q^4)$ are simple, as proved in \cite[Proposition 2.3]{BoPol} and \cite[Section 4]{CsMP}.
\end{remark}

We now recall the classification of $\fq$-linear blocking sets in $\PG(2,q^4)$ of R\'edei type.

\begin{theorem}{\rm (\cite[Section 5]{BoPol})}\label{cor:BoPol}
If $L_U$ is an $\fq$-linear blocking set of $\PG(2,q^4)$ of R\'edei type, then $L_U$ is $\mathrm{P}\Gamma\mathrm{L}$-equivalent to one of the following:
\begin{itemize}
\item [(A)] Baer subplane, which has order $q^4+q^2+1$;
\item [($B_1$)] $L_{B,1}=\{\la(\alpha,x,\mathrm{Tr}_{q^4/q}(x))\ra_{\mathbb{F}_{q^4}} \colon x \in \mathbb{F}_{q^4}, \alpha\in\fq,(\alpha,x)\ne(0,0)\}$, which has order $q^4+q^3+1$;
\item [($B_{21}$)] $L_{B,21}=\{\la(\alpha,x,x^q-x^{q^3})\ra_{\mathbb{F}_{q^4}} \colon x \in \mathbb{F}_{q^4}, \alpha\in\fq,(\alpha,x)\ne(0,0)\}$, which has order $q^4+q^3+q^2+1$;
\item [($B_{22}$)] $L_{B,22}=\{\la(-\xi x_0+x_1,-\eta x_2+x_3,x_4)\ra_{\mathbb{F}_{q^4}} \colon (x_0,\ldots,x_4)\in\fq^5\setminus\{\mathbf{0}\}\}$ for all $\eta \in \mathbb{F}_{q^4}\setminus\mathbb{F}_{q^2}$ and some $\xi \in \mathbb{F}_{q^2}\setminus\fq$, which has order $q^4+q^3+q^2-q+1$;
\item [($C_{11}$)] $L_{C,11}=\{\la(\alpha,x,x^q)\ra_{\mathbb{F}_{q^4}} \colon x \in \mathbb{F}_{q^4}, \alpha\in\fq,(\alpha,x)\ne(0,0)\}$, which has order $q^4+q^3+q^2+q+1$;
\item [($C_{12}$)] $L_{C,12}=\{\la(\eta x_0-\eta^2 x_1+x_2,-\eta x_1 +x_3,x_4)\ra_{\mathbb{F}_{q^4}} \colon (x_0,\ldots,x_4)\in\fq^5\setminus\{\mathbf{0}\} \}$ for all $\eta \in \mathbb{F}_{q^4}\setminus \mathbb{F}_{q^2}$, which has order $q^4+q^3+1$;
\item [($C_{13}$)] $L_{C,13}=\{\la(x_0+\eta_1x_2+x_3,x_1+\eta_2^{-1}x_3,x_4)\ra_{\mathbb{F}_{q^4}} \colon (x_0,\ldots,x_4)\in\fq^5\setminus\{\mathbf{0}\} \}$ for all $\eta_1,\eta_2\in \mathbb{F}_{q^4}\setminus\mathbb{F}_{q^2}$ such that $1,\eta_1,\eta_2,-\eta_1\eta_2$ are $\fq$-linearly independent, which has order $q^4+q^3+q^2-q+1$;
\item [($C_{14}$)] $L_{C,14}=\{\la(x_0-(d\eta_1+\eta_2)x_2+\eta_1 x_3,x_1+c\eta_1 x_2+\eta_2 x_3,x_4)\ra_{\mathbb{F}_{q^4}} \colon (x_0,\ldots,x_4)\in\fq^5\setminus\{\mathbf{0}\} \}$ for some $c,d \in \mathbb{F}_q$ such that $f(x,y)=y^2-cx^2-dxy$ is irreducible over $\fq$ and some $\eta_1,\eta_2\in \mathbb{F}_{q^4}\setminus\mathbb{F}_{q^2}$ such that $1,\eta_1,\eta_2,f(\eta_1,\eta_2)$ are $\fq$-linearly independent and $f(\eta_1^{q^i}-\eta_1,\eta_2^{q^i}-\eta_2)\ne 0$ for each $i \in \{1,2\}$, which has order $q^4+q^3+q^2+q+1$;
\item [($C_{15}$)] $L_{C,15}=\{\la(x_1-\eta_1x_3,-\eta_1 x_0+x_2-\eta_2 x_3,x_4)\ra_{\mathbb{F}_{q^4}} \colon (x_0,\ldots,x_4)\in\fq^5\setminus\{\mathbf{0}\} \}$ for all $\eta_1 \in \mathbb{F}_{q^4}\setminus \mathbb{F}_{q^2}$ and all $\eta_2 \in \mathbb{F}_{q^4}$ such that $1,\eta_1,\eta_2,-\eta_1^2$ are $\fq$-linearly independent, which has order $q^4+q^3+q^2+1$.
\end{itemize}
\end{theorem}

We also recall the possible weight distributions of $\fq$-linear sets of rank $4$ in $\PG(1,q^n)$.

\begin{proposition}{\rm (\cite[Result 1.13]{DeBoeckVdV2020}, see also \cite[Lemma 10]{LavVdV2010})}\label{prop:rank4distribution}
Let $L_U$ be an $\fq$-linear set of rank $4$ in $\PG(1,q^n)$ with $n\geq4$. Then one of the following holds:
\begin{itemize}
    \item $|L_U|=1$, $L_U$ is a unique point of weight $4$;
    \item $|L_U|=q^2+1$, all points of $L_U$ have weight $2$ and $L_U$ is a Baer subline of $\PG(1,q^n)$;
    \item $|L_U|=q^3+1$, $L_U$ has $1$ point of weight $3$ and $q^3$ points of weight $1$;
    \item $|L_U|=q^3+1$, $L_U$ has $q+1$ points of weight $2$ and $q^3-q$ points of weight $1$;
    \item $|L_U|=q^3+q^2-q+1$, $L_U$ has $2$ points of weight $2$ and $q^3+q^2-q-1$ points of weight $1$;
    \item $|L_U|=q^3+q^2+1$, $L_U$ has $1$ point of weight $2$ and $q^3+q^2$ points of weight $1$;
    \item $|L_U|=q^3+q^2+q+1$, all points of $L_U$ have weight $1$.
\end{itemize}
\end{proposition}

As a corollary we get the complete classification of $\fq$-linear sets of rank $4$ in $\PG(1,q^4)$ and their weight distribution.

\begin{corollary}\label{cor:PG1q4}
Let $L_U$ be an $\fq$-linear set of rank $4$ in $\PG(1,q^4)$. 
\begin{itemize}
    \item[(o)] If $|L_U|=q^2+1$, then all points have weight $2$ and $L_U$ is a Baer subline of ${\rm PG}(1,q^4)$.
    \item [(i)] If $|L_U|=q^3+q^2+q+1$, then all points have weight $1$ and $L_U$ is $\mathrm{P}\Gamma\mathrm{L}$-equivalent to 
            \[ \{\langle (x,x^q+\delta x^{q^3}) \rangle_{\mathbb{F}_{q^4}} \colon x \in \mathbb{F}_{q^4}^*\}, \]
            for some $\delta \in \mathbb{F}_{q^4}$ with $\N_{q^4/q}(\delta)\ne 1$.
    \item [(ii)] If $|L_U|=q^3+1$, then $L_U$ is $\mathrm{P}\Gamma\mathrm{L}$-equivalent either to
            \begin{equation}\label{eq:iiA} \{\langle (x,\mathrm{Tr}_{q^4/q}(x)) \rangle_{\mathbb{F}_{q^4}} \colon x \in \mathbb{F}_{q^4}^*\} \end{equation}
            with $1$ point of weight $3$ and $q^3$ points of weight $1$; or to
            \begin{equation}\label{eq:iiB} \{\la(\eta x_0-\eta^2 x_1+x_2,-\eta x_1 +x_3)\ra_{\mathbb{F}_{q^4}} \colon (x_0,\ldots,x_3)\in\fq^4\setminus\{\mathbf{0}\} \} \end{equation}
            for all $\eta \in \mathbb{F}_{q^4}\setminus \mathbb{F}_{q^2}$, with $q+1$ points of weight $2$ and $q^3-q$ points of weight $1$.
    \item [(iii)] If $|L_U|=q^3+q^2+1$, then $L_U$ has $1$ point of weight $2$ and $q^3+q^2$ points of weight $1$; $L_U$ is $\mathrm{P}\Gamma\mathrm{L}$-equivalent either to 
            \[\{\la(x,x^q-x^{q^3})\ra_{\mathbb{F}_{q^4}} \colon x \in \mathbb{F}_{q^4}^*\},\]
            or to 
            \[ \{\la(x_1-\eta_1x_3,-\eta_1 x_0+x_2-\eta_2 x_3)\ra_{\mathbb{F}_{q^4}} \colon (x_0,\ldots,x_3)\in\fq^4\setminus\{\mathbf{0}\} \}\] 
            for all $\eta_1 \in \mathbb{F}_{q^4}\setminus \mathbb{F}_{q^2}$ and $\eta_2 \in \mathbb{F}_{q^4}$ such that $1,\eta_1,\eta_2,\eta_1^2$ are $\fq$-linearly independent.
    \item [(iv)] If $|L_U|=q^3+q^2-q+1$, then $L_U$ has $2$ points of weight $2$ and $q^3+q^2-q-1$ points of weight $1$; $L_U$ is $\mathrm{P}\Gamma\mathrm{L}$-equivalent either to
            \[\{\la(-\xi x_0+x_1,-\eta x_2+x_3)\ra_{\mathbb{F}_{q^4}} \colon (x_0,\ldots,x_3)\in\fq^4\setminus\{\mathbf{0}\} \}\] 
            for all $\eta \in \mathbb{F}_{q^4}\setminus\mathbb{F}_{q^2}$ and some $\xi \in \mathbb{F}_{q^2}\setminus\fq$, or to
            \[ \{\la(x_0+\eta_1x_2+x_3,x_1+\eta_2^{-1}x_3)\ra_{\mathbb{F}_{q^4}} \colon (x_0,\ldots,x_3)\in\fq^4\setminus\{\mathbf{0}\} \}\] 
            for all $\eta_1,\eta_2\in \mathbb{F}_{q^4}\setminus\mathbb{F}_{q^2}$ such that $1,\eta_1,\eta_2,\eta_1\eta_2$ are $\fq$-linearly independent.
\end{itemize}
\end{corollary}
\begin{proof}
Assume that $L_U$ is contained in the line $\ell=\PG(1,q^4)\subset \PG(2,q^4)$, and is not a Baer subline of $\ell$.
Consider $\mathbf{v}\in \overline{V}\setminus V$, then $L_{\langle U,\mathbf{v}\rangle_{\fq}}$ is an $\fq$-linear blocking set in $\PG(2,q^4)$.
By Theorem \ref{cor:BoPol}, $L_{\langle U,\mathbf{v}\rangle_{\fq}}$ is $\mathrm{P}\Gamma\mathrm{L}$-equivalent to one of $L_{B,1}$, $L_{B,21}$, $L_{B,22}$, $L_{C,11}$, $L_{C,12}$, $L_{C,13}$, $L_{C,14}$ and $L_{C,15}$ as defined in Theorem \ref{cor:BoPol}.
By \cite[Theorem 5]{LunPol2000}, the $\fq$-linear blocking sets with more than one R\'edei line are exactly the ones which are $\mathrm{P}\Gamma\mathrm{L}$-equivalent to $L_{B,1}$. All the remaining $\fq$-linear blocking sets have a unique R\'edei line; in this case $L_U$, being simple as a linear set, is $\mathrm{P}\Gamma\mathrm{L}$-equivalent to the intersection of one of the blocking sets $L_{B,21}$, $L_{B,22}$, $L_{C,11}$, $L_{C,12}$, $L_{C,13}$, $L_{C,14}$ and $L_{C,15}$ with its R\'edei line.
The explicit shape and the size of $L_U$ are obtained using \cite[Theorem 3.4]{CsZ2018} for the blocking sets $L_{C,11}$, $L_{C,14}$, and using the explicit form of the intersection with their R\'edei lines for the remaining blocking sets.

The weight distribution of $L_U$ follows from Proposition \ref{prop:rank4distribution}. In fact, if $|L_U|\ne q^3+1$, then the size of $L_U$ determines uniquely the weight distribution of $L_U$. If $|L_U|=q^3+1$, then $L_U$ is ${\rm P\Gamma L}$-equivalent to one of the linear sets \eqref{eq:iiA} and \eqref{eq:iiB}, in which the weight of the point $\langle(1,0)\rangle_{\mathbb{F}_{q^4}}$ is respectively $3$ and $2$; the claim follows.
\end{proof}

In particular, this classifies all the clubs of rank $4$ of $\mathrm{PG}(1,q^4)$.

\begin{corollary}\label{cor:PG1q4clubs}
Any $3$-club of rank $4$ of $\PG(1,q^4)$ is $\mathrm{P}\Gamma\mathrm{L}$-equivalent to 
\[ \{\langle (x,\mathrm{Tr}_{q^4/q}(x)) \rangle_{\mathbb{F}_{q^4}} \colon x \in \mathbb{F}_{q^4}^*\}. \]
Any $2$-club of rank $4$ of $\PG(1,q^4)$ is $\mathrm{P}\Gamma\mathrm{L}$-equivalent either to
\[ \{\la(x,x^q-x^{q^3})\ra_{\mathbb{F}_{q^4}} \colon x \in \mathbb{F}_{q^4}^*\}, \]
or to
\[ \{\la(x_1-\eta_1x_3,-\eta_1 x_0+x_2-\eta_2 x_3)\ra_{\mathbb{F}_{q^4}} \colon (x_0,\ldots,x_3)\in\fq^4\setminus\{\mathbf{0}\} \}\]
for all $\eta_1\in\mathbb{F}_{q^4}$ and $\eta_2\in\mathbb{F}_{q^4}$ such that $1,\eta_1,\eta_2,\eta_1^2$ are $\fq$-linearly independent.
\end{corollary}
\begin{comment}
\begin{proof}
The first part follows by \cite[Theorem 2.3]{DeBoeckVdV}.
Consider a $2$-club $L_U$ of $\PG(1,q^4)$, then it has size $q^3+q^2+1$.  By the point (iii) of Corollary \ref{cor:PG1q4}, $L_U$ is $\mathrm{P}\Gamma\mathrm{L}$-equivalent either to
\[\{\la(x,x^q-x^{q^3})\ra_{\mathbb{F}_{q^4}} \colon x \in \mathbb{F}_{q^4}^*\},\]
or to 
\[ \{\la(x_1-\eta_1x_3,-\eta_1 x_0+x_2-\eta_2 x_3)\ra_{\mathbb{F}_{q^4}} \colon (x_0,\ldots,x_3)\in\fq^4\setminus\{\mathbf{0}\} \}.\]
The assertion then follows noting that the latter has at least two points of weight two, namely $\la(1,0)\ra_{\mathbb{F}_{q^4}}$ and $\la(0,1)\ra_{\mathbb{F}_{q^4}}$. \textcolor{blue}{No, only $\langle(0,1)\rangle_{\mathbb{F}_{q^4}}$ has weight $2$}
\end{proof}
\end{comment}

\begin{remark}
By Proposition \ref{prop:rank4distribution}, the integers $r$ for which there exists an $r$-fat polynomial over $\mathbb{F}_{q^4}$ are $0$, $1$, $2$ and $q+1$.
\end{remark}

\begin{comment}
\begin{corollary}
Any $1$-fat polynomial over $\F_{q^4}$ is equivalent either to $x^q-x^{q^3}$ or to the trace function $\mathrm{Tr}_{q^4/q}(x)$.

\textcolor{blue}{
Finora, come gia' visto da De Boeck-Van de Voorde: there exists an $r$-fat polynomial over $\mathbb{F}_{q^4}$ if and only if $r\in\{0,1,2,q+1\}$.
}
\end{corollary}

\begin{remark}
Corollary \ref{cor:PG1q4clubs} classifies all the linearized polynomials defining an $i$-club when $n=4$.
This answers to Problem 4 posed in \cite[Section 6]{CsMPq5} when $n=4$.
\end{remark}
\end{comment}

\section{$r$-fat polynomials over \texorpdfstring{$\F_{q^5}$}{Lg}}\label{sec:q5}

Let $f(x)$ be an $r$-fat polynomials over $\mathbb{F}_{q^5}$.
By the Main Theorem in \cite{DeBoeckVdV}, 
\[ r\in\{0,1\}\cup[q-2\sqrt{q}+1,q+2\sqrt{q}+1]\cup\{2q,2q+1,2q+2,3q,3q+1,q^2+1\}. \]
In the following, we consider the case $r=1$.

Let $q=2^{2s}$ with $s\geq 1$, $w\in\mathbb{F}_4\subseteq\fq$ be such that $w^2=w+1$, and
\[ f_w(x)=x^{q^4}+wx^{q^3}+wx^q+x \in\mathbb{F}_{q^5}[x]. \]

\begin{theorem}\label{th:3club}
Let $q=2^{2s}$ with $s\geq 1$. The polynomial $f_w$ is a $1$-fat polynomial over $\mathbb{F}_{q^5}$ of index $0$ and maximum weight $3$.
\end{theorem}
\proof
First we show that $\dim_{\fq}\ker(f_w)=3$. To this aim, consider the Dickson matrix 
\[ M^{(0)}:=\left(
\begin{array}{ccccc}
    1 & w & 0 & w & 1\\
    1 & 1 & w & 0 & w\\
    w & 1 & 1 & w & 0\\
    0 & w & 1 & 1 & w\\
    w & 0 & w & 1 & 1\\
\end{array}
\right)\]
associated with $f_w(x)$.
For any $j\leq5$, let $M_j$ be the North-East principal minor of order $j$ of $M^{(0)}$. Then
\[ M_2=w^2\neq0,\,\, M_3=M_4=M_5=0,\]
and hence $\dim_{\fq}\ker(f_w)=3$ by \cite[Theorem 1.3]{qres}.

Now we prove that any $m\in \mathbb{F}_{q^5}^*$ satisfies $\dim_{\fq}\ker(f_w(x)+mx)\leq1$. To this aim, consider the Dickson matrix
\[M^{(m)}:=\left(
\begin{array}{ccccc}
    1+m & w & 0 & w & 1\\
    1 & 1+m^q & w & 0 &w\\
    w & 1 & 1+m^{q^2} & w & 0\\
    0 & w & 1 & 1+m^{q^3} & w\\
    w & 0 & w & 1 & 1+m^{q^4}\\
\end{array}
\right)\]
associated with $f_w(x)+mx$.
For simplicity, let $m_i=m^{q^i}$ for any $i=0,\ldots,4$. 
Let $M_4=m_1m_2m_3 + w m_1m_2 + m_1m_3 + wm_2m_3$ be the North-East principal minor of order $4$ of $M^{(m)}$, and suppose that $M_4=0$. 

If $m_1 m_2 + m_1 + w m_2= 0$, then $m_1m_2=0$, that is $m=0$, a contradiction.
Then $m_1 m_2 + m_1 + w m_2\neq 0$, whence $m_3=\frac{w m_1 m_2}{m_1 m_2 + m_1 + w m_2}$. Now,
    \begin{eqnarray*}
    m_4&=&(m_3)^q= \frac{m_1 m_2}{m_1 m_2 + w^2 m_1 + w^2 m_2},\\
    m_0&=&(m_4)^q=\frac{w m_1 m_2}{m_1 m_2 + w m_1 + m_2}.
    \end{eqnarray*}
    Clearly, all the denominators are nonzero, otherwise $m=0$. 
    The minor $M_5$ of order $5$ of $M^{(m)}$ equals
    \begin{eqnarray*}
    M_5&=& m_0 m_1 m_2 m_3 m_4 + m_0 m_1 m_2 m_3 + m_0 m_1 m_2 m_4 + w^2 m_0 m_1 m_2\\
    &&+ m_0 m_1 m_3 m_4 + m_0 m_1 m_3 + w^2 m_0 m_1 m_4 + m_0 m_2 m_3 m_4\\
    &&+ m_0 m_2 m_3 + m_0 m_2 m_4 + w^2 m_0 m_3 m_4 + m_1 m_2 m_3 m_4\\
    &&+w^2 m_1 m_2 m_3 + m_1 m_2 m_4 + m_1 m_3 m_4 + w^2 m_2 m_3 m_4,
    \end{eqnarray*}
    which can be rewritten in $m_1$ and $m_2$ as follows
    $$\frac{m_1^4 m_2^4}{(m_1 m_2 + w^2 m_1 + w^2 m_2)(m_1 m_2 + w m_1 + m_2)(m_1 m_2 + m_1 + w m_2)}.$$
    It is readily seen that $M_5=0$ yields $m=0$, a contradiction. Summing up, if $m\ne0$ then $M_4$ and $M_5$ cannot vanish at the same time and therefore $\dim_{\fq}\ker(f_w(x)+mx)\leq1$ by \cite[Theorem 1.3]{qres}.

Therefore $f_w(x)$ is a $1$-fat polynomial over $\F_{q^5}$ of index $0$ and maximum weight $3$.
\endproof

\begin{remark}
Similar arguments to the ones used in the proof of Theorem \ref{th:3club} show that $f_w(x)$ is a $1$-fat polynomial of index $t$ and maximum weight $3$ whenever $t \in\{0,1,3,4\}$.
\end{remark}

\begin{corollary}
Let $q=2^{2s}$ with $s\geq 1$. The linear set $L_{f_w}$ is a $3$-club in $\PG(1,q^5)$.
\end{corollary}

Now we investigate some properties of $L_{f_w}$.

\begin{proposition}\label{prop:simple}
Let $q=2^{2s}$ with $s\geq1$, and $g(x)\in\mathbb{F}_{q^5}[x]$ be such that $L_{g}=L_{f_w}$.
Then there exists $\lambda\in\mathbb{F}_{q^5}^*$ such that
$\lambda U_{g}\in\{U_{f_w},U_{\hat{f}_w}\}$.
Furthermore, $L_{f_w}$ is simple.
\end{proposition}

\begin{proof}
By Theorem \ref{th:3club}, $L_{g}$ is a $3$-club.
Therefore $L_{g}$ is not of pseudoregulus type; see \cite[Remark 5.6]{CsMP}.
By \cite[Theorem 1.4]{CsMPq5}, there exists $\lambda\in\mathbb{F}_{q^5}$ such that $g(x)=f_w(\lambda x)/\lambda$ or $g(x)=\hat{f}_w(\lambda x)/\lambda$, and hence $\lambda U_{g}\in\{U_{f_w},U_{\hat{f}_w}\}$.
Since
\[ \begin{pmatrix} 1 & w+1 \\ 0 & 1 \end{pmatrix}\cdot U_{f_w} =U_{\hat{f}_w}, \]
$U_{f_w}$ and $U_{\hat{f}_w}$ are ${\rm \Gamma L}(2,q^5)$-equivalent.
The claim then follows by the first part of the proof.
\end{proof}

\begin{proposition}
Let $q=2^{2s}$ with $s\geq 1$. %, and consider $w\in \mathbb{F}_{4}\subseteq \mathbb{F}_{2^{2s}}$ such that $w^2=w+1$.
The stabilizer of $U_{f_w}$ in $\mathrm{\Gamma L}(2,q^5)$ is
\[ \left\{ \begin{pmatrix} a & b\\ 0 & a \end{pmatrix} \colon a \in \fq^*,\,b^{q^2}+(w+1)b^q+b=0 \right\}\rtimes\langle\Phi^2 \rangle \cong (E_{q^2}\times C_{q-1})\rtimes C_{5s}, \]
where $\Phi:(x,y)\mapsto (x^2,y^2)$.
%The setwise stabilizer of $\{U_{f_w},U_{\hat{f}_w}\}$ in $\rm \Gamma L(2,q^5)$ is
%\[ \left\{ \begin{pmatrix} a & b\\ 0 & a \end{pmatrix} \colon a \in \fq^*,\,b^{2q^2}+wb^{q^2}+wb^{2q}+b^q+b^2+wb=0 \right\}\rtimes\langle\Phi^2 \rangle \cong (E_{2q^2}\times C_{q-1})\rtimes C_{5s}. \]
The stabilizer of $L_{f_w}$ in ${\rm P\Gamma L}(2,q^5)$ is
\[ \left\{ \begin{bmatrix} 1 & b\\ 0 & 1 \end{bmatrix} \colon b^{2q^2}+wb^{q^2}+wb^{2q}+b^q+b^2+wb=0 \right\}\rtimes\langle\Phi^2 \rangle \cong E_{2q^2}\rtimes C_{5s}. \]
\end{proposition}

\begin{proof}
Let $a,b,c,d\in\mathbb{F}_{q^5}$ with $ad-bc\ne0$ and $\sigma\in{\rm Aut}(\mathbb{F}_{q^5})$ be such that
$\begin{pmatrix} a & b\\ c & d \end{pmatrix}\cdot U_{f_w}^{\sigma}=U_{f_w}$.
This happens if and only if for each $x\in \mathbb{F}_{q^5}$ there exists $z \in \mathbb{F}_{q^5}$ such that
\begin{equation}\label{eq:uguaglia}
\begin{pmatrix} a & b\\ c & d \end{pmatrix} \begin{pmatrix} x^{\sigma} \\ x^{\sigma q^4}+w^{\sigma}x^{\sigma q^3}+w^{\sigma}x^{\sigma q}+x^{\sigma} \end{pmatrix}=\begin{pmatrix} z \\ z^{q^4}+wz^{q^3}+wz^q+z \end{pmatrix}.
\end{equation}
As the point $\la(1,0)\ra_{\mathbb{F}_{q^5}}$ is the only point of weight $3$ in $L_{f_w}$, it follows that $c=0$.
Thus, \eqref{eq:uguaglia} is equivalent to
\begin{eqnarray*}
d(x^{\sigma q^4}+w^{\sigma}x^{\sigma q^3}+w^{\sigma}x^{\sigma q}+x^{\sigma})&=&a^{q^4}x^{\sigma q^4}+b^{q^4}(x^{\sigma q^3}+w^{\sigma}x^{\sigma q^2}+w^{\sigma}x^{\sigma}+x^{\sigma q^4})\\ 
&&+w[a^{q^3}x^{\sigma q^3}+b^{q^3}(x^{\sigma q^2}+w^{\sigma}x^{\sigma q}+w^{\sigma}x^{\sigma q^4}+x^{\sigma q^3})]\\
&&+w[a^qx^{\sigma q}+b^q(x^\sigma +w^\sigma x^{\sigma q^4}+w^\sigma x^{\sigma q^2}+x^{\sigma q})]\\
&&+ax^\sigma+b(x^{\sigma q^4}+w^{\sigma}x^{\sigma q^3}+w^\sigma x^{\sigma q}+x^{\sigma})\end{eqnarray*}
for every $x\in\mathbb{F}_{q^5}$. So
\begin{equation}\label{eq:system}
\left\{
\begin{array}{lllll} 
d=w b^{q^4}+w b^q+a+b,\\
wd= w^2 b^{q^3}+w a^q+ w b^q+wb,\\
0=w b^{q^4}+wb^{q^3}+w^2b^q,\\
wd=b^{q^4}+wa^{q^3}+wb^{q^3}+wb,\\
d=a^{q^4}+b^{q^4}+w^2 b^{q^3}+w^2 b^q+b
\end{array}
\right.
\end{equation}
if $w^{\sigma}=w$, and
\begin{equation}\label{eq:system3}
\left\{
\begin{array}{lllll} 
d=a+(w+1)b^{q^4}+wb^q+b,\\
(w+1)d= wa^q+b^{q^3}+wb^q+(w+1)b,\\
0=(w+1)b^{q^4}+wb^{q^3} + b^{q},\\
(w+1)d=w a^{q^3}+b^{q^4}+w b^{q^3} + (w+1)b,\\
d=a^{q^4}+b^{q^4}+ b^{q^3}+ b^q+b
\end{array}
\right.
\end{equation}
if $w^{\sigma}=w+1$.

Consider first System \eqref{eq:system}.
It is easy to see that $d=a$ and $a \in \F_q^*$. Furthermore, $b$ satisfies the following system:
%By the third equation we get
%\begin{equation}\label{eq:cond1}
%b+b^q+w b^{q^3}=0.
%\end{equation}
%Using \eqref{eq:cond1}, we get in the second equation of System \eqref{eq:system} $d=a^q$.
%By using that $w^2=w+1$ and \eqref{eq:cond1}, the last equation of System \eqref{eq:system} becomes $d=a^{q^4}$, that is $a \in \fq^*$ and $d=a$. So,  using the first and fourth equation of \eqref{eq:system} together with \eqref{eq:cond1}, System \eqref{eq:system} reads 
\begin{equation}\label{eq:system2}
\left\{
\begin{array}{lllll} 
w b^{q^4}+w b^q+b = 0,\\
b^{q^4}+wb^{q^3}+wb = 0,\\
w b^{q^3} + b^q + b = 0.
\end{array}
\right.
\end{equation}
From the second and third equation in \eqref{eq:system2} we get
\begin{equation}\label{eq:cond2}
    b^{q^2}= (w+1)b^q+b.
\end{equation}
By direct computation using $w^2=w+1$, all the $q^2$ distinct solutions of \eqref{eq:cond2} are in $\mathbb{F}_{q^5}$ and every $b$ satisfying \eqref{eq:cond2} satisfies also \eqref{eq:system2}.

Consider System \eqref{eq:system3}.
Similar arguments show that $a=d=0$, a contradiction.

Therefore, ${\rm \Gamma L}(2,q^5)_{U_{f_w}}$ equals
\[ \left\{ \begin{pmatrix} a & b\\ 0 & a \end{pmatrix} \colon a \in \fq^*,\,b^{q^2}+(w+1)b^q+b=0 \right\}\rtimes\langle\Phi^2 \rangle \cong (E_{q^2}\times C_{q-1})\rtimes C_{5s}, \]
where $\Phi:(x,y)\mapsto (x^2,y^2)$.
As already noticed in the proof of Proposition \ref{prop:simple}, the involution $\tau=\begin{pmatrix} 1 & w+1 \\ 0 & 1 \end{pmatrix}$ maps $U_{f_w}$ to $U_{\hat{f}_w}$.
Since $\tau$ commutes with every element of ${\rm \Gamma L}(2,q^5)_{U_{f_w}}$ and maps $U_{f_w}$ to $U_{\hat{f}_w}$,
we have that ${\rm \Gamma L}(2,q^5)_{U_{f_w}}$ stabilizes $U_{\hat{f}_w}$.
Thus, the setwise stabilizer ${\rm \Gamma L}(2,q^5)_{\{U_{f_w},U_{\hat{f}_w}\}}$ is
\[
{\rm \Gamma L}(2,q^5)_{\{U_{f_w},U_{\hat{f}_w}\}}=\langle {\rm \Gamma L}(2,q^5)_{U_{f_w}},\tau\rangle =
{\rm \Gamma L}(2,q^5)_{U_{f_w}}\times\langle\tau\rangle =
\]
\[
\left\{ \begin{pmatrix} a & b\\ 0 & a \end{pmatrix}, \begin{pmatrix} a & b+w+1 \\ 0 & a \end{pmatrix} \colon a \in \fq^*,\,b^{q^2}+(w+1)b^q+b=0 \right\}\rtimes\langle\Phi^2 \rangle,
\]
%Now, the polynomial whose roots are 
%\[\{b,w+1+b\colon\; b^{q^2}+(w+1)b^q+b=0\}\]
%is
%\[
%F(T)=(T^{q^2}+(w+1)T^q+T)^2+w(T^{q^2}+(w+1)T^q+T)\]
%\[
%=T^{2q^2}+wT^{q^2}+wT^{2q}+T^q+T^2+wT.
%\]
%Then ${\rm \Gamma L}(2,q^5)_{\{U_{f_w},U_{\hat{f}_w}\}}$ equals
which results to be equal to
\[ \left\{ \begin{pmatrix} a & b\\ 0 & a \end{pmatrix} \colon a \in \fq^*,\,b^{2q^2}+wb^{q^2}+wb^{2q}+b^q+b^2+wb=0 \right\}\rtimes\langle\Phi^2 \rangle . \]

Let $\psi\in{\rm \Gamma L}(2,q^5)$ be such that the projectivity $\Psi\in{\rm P\Gamma L}(2,q^5)$ induced by $\psi$ stabilizes $L_{f_w}$, and let $U_g=\psi(U_{f_w})$.
By Proposition \ref{prop:simple}, there exists $\lambda\in\mathbb{F}_{q^5}^*$ such that either $\delta\psi$ stabilizes $U_{f_w}$ or $\delta\psi$ maps $U_{f_w}$ to $U_{\hat{f}_w}$, where $\delta$ is the scalar matrix with $\lambda$ on the diagonal.
Since $\delta\psi$ induces $\Psi$ in ${\rm P\Gamma L}(2,q^5)$, this implies that ${\rm P\Gamma L}(2,q^5)_{L_{f_w}}$ is induced by ${\rm \Gamma L}(2,q^5)_{\{U_{f_w},U_{\hat{f}_w}\}}$.
The claim follows.
\end{proof}

\begin{remark}
For $q=4$, MAGMA computation shows that $L_{f_w}$ is ${\rm P\Gamma L}$-equivalent to the linear set $L_{\lambda}$ in \eqref{eq:exKMnopol} for a certain $\lambda\in\mathbb{F}_{4^5}^*$.
We do not know whether this is true also for $q>4$.
However, the polynomial $f_w$ is a new example of linearized polynomial which defines a $3$-club, see Problem 4 in \cite[Section 6]{CsMPq5}.
\end{remark}

\begin{remark}
Using Proposition \ref{prop:adjoint} (see also \cite[Remark 3.7]{ZiniZulloInt}), it turns out that $L_{f_w,i}$ is ${\rm P\Gamma L}$-equivalent to $L_{f_w}$ for any $i\in\{1,3,4\}$.
\end{remark}

\section{LP-polynomials}\label{sec:LP}

Through this section, consider the two polynomials $f(x)=x+\delta x^{q^{2s}}$ of index $s$ and $g(x)=x^{q^{s(n-1)}}+\delta x^{q^s}$ of index $0$  in $\fqn[x]$, where $s$ satisfies $\gcd(s,n)=1$. The polynomials $f(x)$ and $g(x)$ will be called \emph{LP-polynomials} after Lunardon and Polverino who first introduced in \cite{LunPol2000} the $\fq$-linear set $L_{f,s}=L_{g,0}$ of $\PG(1,q^n)$.
Note that the polynomials $f(x)$ and $g(x)$ are equivalent; also, it is easily seen that 
\begin{equation}\label{eq:fg} 
\dim_{\F_q} \ker (f(x)-mx^{q^s})=\dim_{\fq}\ker(g(x)-mx). 
\end{equation}
Let $r\geq0$ be such that $f$ is an $r$-fat polynomial of index $s$ over $\fqn$.
Equation \eqref{eq:fg} implies that $g$ is an $r$-fat polynomial of index $0$ over $\fqn$.

In this section we determine $r$, depending on the choice of $\delta$.

%\begin{remark}
%Let $f(x)=x+\delta x^{q^{2s}},g(x)=x^{q^{s(n-1)}}+\delta x^{q^s}\in\fqn[x]$.
%Then the linear sets
%\[
%L_f=\{\langle (x^{q^s},x+\delta x^{q^{2s}}) \rangle_{\fqn} \colon x \in \fqn^*\},\quad L_g=\{\langle (x,x^{q^{s(n-1)}}+\delta x^{q^s}) \rangle_{\fqn} \colon x \in \fqn^*\}
%\]
%coincide and, for any point $P\in{\rm PG}(1,q^n)$, the weight of $P$ in $L_f$ equals the weight of $P$ in $L_g$. 
%\end{remark}

%By \cite[Theorem 5]{Gow}, the next result follows.

\begin{proposition}{\rm (see \cite[Theorem 5]{Gow})}
The maximum weight of $f(x)$ and $g(x)$ is at most two, for any $\delta\in\fqn$.
\end{proposition}
%\begin{proof}
%By Remark \ref{rk:weight}, the weight of the point $\langle (1,m) \rangle_{\fqn}$ is given by $\dim_{\fq} \ker(x+\delta x^{q^{2s}}-mx^{q^s})$, which is smaller than two as $\gcd(s,n)=1$.
%\end{proof}

After a series of papers, Zanella in \cite{Zanella} characterized those $\delta \in \fqn$ for which $f(x)$ (and hence $g(x)$) is a $0$-fat polynomial.

\begin{theorem}{\rm \cite[Theorem 3.4]{Zanella}}\label{th:Zanella}
Let $\delta\in\mathbb{F}_{q^n}$, $\gcd(s,n)=1$, $f(x)=x+\delta x^{q^{2s}}\in\mathbb{F}_{q^n}[x]$.
The polynomial $f$ is $0$-fat if and only if $\N_{q^n/q}(\delta)\ne 1$.
\end{theorem}

Next result summarizes what is known for LP-polynomials when $\mathrm{N}_{q^n/q}(\delta)=1$.

\begin{theorem}
Let $\delta\in\mathbb{F}_{q^n}$, $\gcd(s,n)=1$, $f(x)=x+\delta x^{q^{2s}}\in\mathbb{F}_{q^n}[x]$.

\begin{itemize}
    \item If $n$ and $q$ are both odd and $\N_{q^n/q}(\delta)=1$, then $f$ is a $\frac{q^{n-1}-1}{q^2-1}$-fat polynomial; see \cite[Lemma 4.4]{LMPT2015}.
    \item If $n=4$ and $\N_{q^n/q}(\delta)=1$, then $f$ is an $r$-fat polynomial with
    \[ r=\left\{ \begin{array}{ll} 1 & \mbox{if}\quad \N_{q^4/q^2}(\delta)=1,\\ q+1 & \mbox{if}\quad\N_{q^4/q^2}(\delta)\ne 1; \end{array}\right.\]
    see \cite[Theorem 2.1]{CsZ2018}.
\end{itemize}
\end{theorem}

Following the proof of \cite[Lemma 4.4]{LMPT2015} we determine $r$ for any $q$ when $n$ is odd.

\begin{proposition}\label{prop:nodd}
Let $n$ be an odd integer, $\delta\in\mathbb{F}_{q^n}$, $\gcd(s,n)=1$, $f(x)=x+\delta x^{q^{2s}}\in\mathbb{F}_{q^n}[x]$.
If $\N_{q^n/q}(\delta)=1$ then $f$ is an $\frac{q^{n-1}-1}{q^2-1}$-fat polynomial.
\end{proposition}
\begin{proof}
It is enough to prove that the number of points having weight two in $L_{f,s}$ is $\frac{q^{n-1}-1}{q^2-1}$.
The point $\langle (x_0,x_0^{q^{s(n-1)}}+\delta x_0^{q^{s}})\rangle_{\fqn}$ has weight two in $L_{f,s}$ if and only there exists $\lambda_0 \in \fqn\setminus \fq$ such that $y=\lambda_0 x_0$ and 
\[ \lambda_0(x_0^{q^{s(n-1)}}+\delta x_0^{q^{s}})=(\lambda_0 x_0)^{q^{s(n-1)}}+\delta (\lambda_0 x_0)^{q^{s}}, \]
that is
\begin{equation}\label{eq:cond} x_0^{q^{2s}-1}=\frac{1}{\delta^{q^s}(\lambda_0-\lambda_0^{q^s})^{q^s-1}}. \end{equation}
\begin{comment}
and \textcolor{blue}{hence
\[ \langle (x_0,x_0^{q^{s(n-1)}}+\delta x_0^{q^{s}})\rangle_{\fqn} \cap \{ (x,x^{q^{s(n-1)}}+\delta x^{q^{s}}) \colon x \in \fqn^*\}=\] 
\[\{ (\alpha x_0+\beta \lambda_0 x_0,(\alpha x_0+\beta \lambda_0 x_0)^{q^{s(n-1)}}+\delta (\alpha x_0+\beta \lambda_0 x_0)^{q^{s}} ) \colon \alpha,\beta \in \fq \}.  \]QUESTA PARTE A COSA SERVE?}
\end{comment}
It follows that the number of pairs $(\lambda,x)\in \fqn\setminus \fq \times \fqn^*$ corresponding to a point of weight two in $L_{f,s}$ is the size of the following set
\[ \Omega= \{ (\lambda,x)\in (\fqn\setminus \fq) \times \fqn^* \colon (\lambda,x)\,\, \mbox{satisfies}\,\, \eqref{eq:cond} \}. \]
For any $\overline{\lambda}\in\fqn\setminus\fq$, the number of $\overline{x}\in\fqn^*$ such that $(\overline{\lambda},\overline{x})\in\Omega$ is $q-1$, as $n$ is odd and $\gcd(s,n)=1$.
Therefore $|\Omega|=(q-1)(q^n-q)$.

Recall that $g(x)=x^{q^{s(n-1)}}+\delta x^{q^s}$.
Consider $\lambda,\mu \in \fqn\setminus \fq$. By direct computation, the two equations $g(\lambda x)=\lambda g(x)$ and $g(\mu x)=\mu g(x)$ have a common solution $x\in\mathbb{F}_{q^n}^*$ if and only if $\mu \in \langle 1,\lambda \rangle_{\fq}$, and in this case the two equations have the same solutions.
Thus, the number of $x\in\fqn^*$ such that the point $\langle(x,g(x))\rangle_{\fqn}$ has weight two in $L_{g,0}=L_{f,s}$ is
\[ \frac{|\Omega|}{q^2-q}=q^{n-1}-1,  \]
so that the number of points having weight two in $L_{f,s}$ is $\frac{q^{n-1}-1}{q^2-1}$.
\end{proof}

Let $n$ and $s$ be positive integers such that $\gcd(s,n)=1$ and $q$ be any a prime power.
Consider $F(x)=a_0x+a_1x^{q^s}+\ldots+a_{n-1}x^{q^{s(n-1)}}\in \mathbb{F}_{q^n}[x]$ and the Dickson matrix 
\[ D_{s,F(x)}=\left(\begin{matrix}
a_0 & a_1 & \cdots & a_{n-1} \\
a_{n-1}^{q^s} & a_0^{q^s} & \cdots & a_{n-2}^{q^s} \\
\vdots & & \vdots \\
a_1^{q^{s(n-1)}} & a_2^{q^{s(n-1)}} & \cdots & a_0^{q^{s(n-1)}}
\end{matrix}\right)\]
associated with $F$.
As a consequence of \cite[Proposition 2.2]{Zanella} and \cite{qres}, the following holds.

\begin{proposition}
Let $F(x)=\sum_{i=0}^{n-1} a_i x^{q^{si}}\in \mathbb{F}_{q^n}[x]$ and
\[ \varphi(x,t)=-F(x)t+F(xt)\in \mathbb{F}_{q^n}[x,t]. \]
For any $x_0 \in \fqn$, the following conditions are equivalent:
\begin{itemize}
    \item $w_{L_F}(\langle (x_0,F(x_0)) \rangle_{\mathbb{F}_{q^n}})\geq 2$;
    \item $\mathrm{rk}(D_{s,\varphi(x_0,t)})\leq n-2$;
    \item the $(n-1)$-th order North-West principal minor of $D_{s,\varphi(x_0,t)}$ is zero.
\end{itemize}
\end{proposition}

Next corollary follows from the above proposition and the fact that the weight of a point in $L_{f,s}$ is at most two.

\begin{corollary}
Let $g(x)=x^{q^{s(n-1)}}+\delta x^{q^s}$ and $\varphi(x,t)=-g(x)t+g(xt)$, with $\N_{q^n/q}(\delta)=1$.
For any $x_0 \in \fqn$, the following conditions are equivalent:
\begin{itemize}
    \item $w_{L_g}(\langle (x_0,g(x_0)) \rangle_{\mathbb{F}_{q^n}})= 2$;
    \item $\mathrm{rk}(D_{s,\varphi(x_0,t)})= n-2$;
    \item the $(n-1)$-th order North-West principal minor of $D_{s,\varphi(x_0,t)}$ is zero.
\end{itemize}
\end{corollary}

As in \cite[Propositions 3.2 and 3.3, Theorem 3.4]{Zanella}, in order to determine the value of $r$, we find the number of points of weight two in $L_{f,s}$ in terms of the number of isotropic vectors of a quadratic form.

\begin{proposition}\label{th:correspondence}
Let $n$ be even and $f(x)=x+\delta x^{q^{2s}}\in\fqn[x]$ with $\N_{q^n/q}(\delta)=1$. 
Then the number of points of weight two in $L_{f,s}$ coincides with the number of nonzero elements $u \in \fqn^*$ such that
\[ \mathrm{Tr}_{q^n/q}(d u^{q^s+1})=0 \]
divided by $q^2-1$, where $d\in \fqn^*$ satisfies $\delta=d^{q^s-1}$.
\end{proposition}
\begin{proof}
It is enough to determine the values of $x\in\fqn$ such that $\varphi(x,t)=-g(x)t+g(xt)$ has kernel of $\fq$-dimension two, i.e. (by \cite[Theorem 1.3]{qres}) the $(n-1)$-th order North-West principal minor $M_{n-1}$ of $D_{s,\varphi(x,t)}$ is zero.
As proved in \cite[Proposition 3.2]{Zanella}, we have 
\[ M_{n-1}= (-1)^{n+1} \sum_{i=0}^{n-1} z^{\frac{q^{si}-1}{q^s-1}}, \]
where $z=\delta x^{q^{s(n-1)}-q^s}$.
Clearly, $M_{n-1}=0$ if and only if 
\begin{equation}\label{eq:z}
\sum_{i=0}^{n-1} z^{\frac{q^{si}-1}{q^s-1}}=0.
\end{equation}
For any $z\in\fqn$ we have
\[
z\left(\sum_{i=0}^{n-1} z^{\frac{q^{si}-1}{q^s-1}}\right)^{q^s} - \sum_{i=0}^{n-1} z^{\frac{q^{si}-1}{q^s-1}}=z^{\frac{q^{s(n-1)}-1}{q^s-1}}-1=z^{\frac{q^{n-1}-1}{q-1}}-1=\N_{q^n/q}(z)-1,
\]
and hence any solution $\overline{z} \in \fqn$ of \eqref{eq:z} satisfies $\N_{q^n/q}(\overline{z})=1$.
\begin{comment}
Indeed, raising \eqref{eq:z} to $q^s$, multiplying by $z$ and then subtracting to Equation \eqref{eq:z} we get $1-\N_{q^n/q}(z)=0$.
\end{comment}
Therefore $z \in \fqn$ is a solution of \eqref{eq:z} if and only if there exists $y\in \fqn$ such that $z=y^{q^s-1}$ and $\mathrm{Tr}_{q^n/q}(y)=0$.
Since $n$ is even, $\gcd(q^{s(n-2)}-1,q^{sn}-1)=q^{2s}-1$, and hence for any $x \in \fqn$ there exists $u \in \fqn$ such that $x^{q^{s(n-1)}-q^s}=u^{q^{2s}-1}$, and conversely.
So, after writing $z=\delta u^{q^{2s}-1}$, Equation \eqref{eq:z} reads
\[ \sum_{i=0}^{n-1}(\delta u^{q^{2s}-1})^{\frac{q^{si}-1}{q^s-1}}=0,\]
i.e.
\[\mathrm{Tr}_{q^n/q}(d u^{q^s+1})=0. \]
The claim follows.
\end{proof}

Clearly, $Q\colon u\in\fqn \mapsto \mathrm{Tr}_{q^n/q}(du^{q^s+1})\in \mathbb{F}_q$ is a quadratic form of the $\fq$-vector space $\mathbb{F}_{q^n}$.
The associated bilinear form is
\[ \sigma \colon (u,v) \in \fqn\times\fqn\mapsto  \mathrm{Tr}_{q^n/q}(d(u^{q^s}v+uv^{q^s}))\in \mathbb{F}_q. \]
Recalling that $\delta=d^{q^s-1}$, Proposition \ref{prop:radical} follows.

\begin{proposition}\label{prop:radical}
The radical ${\rm Rad}(\sigma)$ of the bilinear form $\sigma$ is only the zero vector if $\mathrm{N}_{q^n/q}(\delta)\ne (-1)^{n/2}$ and $\{\delta u^{q^{2s}}+u\colon u \in \fqn\}$ otherwise.
In particular, the rank of $Q$ is $n$ if  $\mathrm{N}_{q^n/q}(\delta)\ne (-1)^{n/2}$ and $n-2$ otherwise. 
\end{proposition}
\begin{proof}
Recall that 
\[ \mathrm{Rad}(\sigma)=\{u \in \fqn \colon \sigma(u,v)=0, \,\, \forall v \in \fqn\}. \]
By direct checking,
\[ \mathrm{Tr}_{q^n/q}(d(u^{q^s}v+uv^{q^s}))=0 \]
if and only if
\[ \mathrm{Tr}_{q^n/q}((du^{q^s}+ (du)^{q^{s(n-1)}})v)=0. \]
Therefore, the elements of $ \mathrm{Rad}(\sigma)$ are the roots in $\fqn$ of $\delta x^{q^{2s}}+x$, which has in $\fqn$ either $1$ root if $\N_{q^n/q}(\delta)\ne(-1)^{n/2}$, or $q^2$ roots if $\N_{q^n/q}(\delta)=(-1)^{n/2}$.
\end{proof}

\begin{proposition}\label{prop:2weightneven}
Let $n$ be even and $\delta \in \fqn$ be such that $\N_{q^n/q}(\delta)=(-1)^{n/2}$.
The number of nonzero isotropic vectors of $Q$ is either 
\[q^2(q^{(n-2)/2-1}+1)(q^{(n-2)/2}-1)+q^2-1\]
or 
\[q^2(q^{(n-2)/2}+1)(q^{(n-2)/2-1}-1)+q^2-1.\]
\end{proposition}
\begin{proof}
By Proposition \ref{prop:radical}, the radical of $\sigma$ has $q^2$ vectors and hence the set of isotropic vectors of $\sigma$ is the direct sum of $\mathrm{Rad}(\sigma)$ and the set of isotropic vectors of the restriction $Q\mid_{V^\prime}$ of $Q$ to an $\fq$-subspace $V'=V(n-2,q)$ of $\fqn$ such that $V'\cap \mathrm{Rad}(\sigma)=\{0\}$.
Note that $Q\mid_{V^\prime}$ has rank $n-2$.
Therefore, the number $N$ of nonzero isotropic vectors of $Q$ is $q-1$ times the size of the cone of $\mathrm{PG}(n-1,q)$ whose vertex is the line ${\rm PG}(\mathrm{Rad}(\sigma),\fq)$ and whose basis is the non-singular quadric $X$ of $\mathrm{PG}(V',\mathbb{F}_q)=\mathrm{PG}(n-3,q)$ associated to $Q\mid_{V^\prime}$.
Thus
\[ N=(q^2|X|+q+1)(q-1), \]
and the claim follows (see \cite[Theorem 5.21]{Hirschfeld}).
\end{proof}

As a consequence, recalling that  nonzero isotropic vectors define points of weight two by Proposition \ref{th:correspondence}, we have the following results.

\begin{corollary}\label{cor:1}
Let $\N_{q^n/q}(\delta)=1=(-1)^{n/2}$, so that $q$ is even or $4\mid n$.
Then $f(x)$ is an $r$-fat polynomial of index $s$ over $\fqn$ with 
\[ r=\frac{q^2(q^{\frac{n-2}2}+1)(q^{\frac{n-2}2-1}-1)}{q^2-1}+1. \]
%\[ x_2 = \begin{cases} 
%\frac{q^2(q^{\frac{n-2}2-1}+1)(q^{\frac{n-2}2}-1)}{q^2-1}+1 & \text{if}\,\, 4 \mid n,\\ 
%\frac{q^2(q^{\frac{n-2}2}+1)(q^{\frac{n-2}2-1}-1)}{q^2-1}+1 & %\text{otherwise}.
%\end{cases} \]
\end{corollary}

%Let $\N_{q^n/q}(\delta)\ne(-1)^{n/2}$, then the quadric associated with the quadratic form $Q$ is irreducible.
%Therefore, we have the following result.

\begin{corollary}\label{cor:2}
Let $\N_{q^n/q}(\delta)=1\ne(-1)^{n/2}$, so that $q$ is odd, $2\mid n$ and $4\nmid n$.
Then $f(x)$ is an $r$-fat polynomial of index $s$ over $\fqn$ with
\[
r=\frac{(q^{n/2}+1)(q^{n/2-1}-1)}{q^2-1}.
\]
%\[ x_2 = \begin{cases} 
%\frac{(q^{n/2}+1)(q^{n/2}-1)}{q^2-1} & \text{if}\,\, 4 \mid n,\\ 
%\frac{(q^{n/2}+1)(q^{n/2-1}-1)}{q^2-1} & \text{otherwise}.
%\end{cases} \]
\end{corollary}
\begin{proof}
The proof follows arguing as in the proof of Proposition \ref{prop:2weightneven}, using that $Q$ has rank $n$ and hence $Q$ has $(q^{n/2}+1)(q^{n/2-1}-1)$ nonzero isotropic vectors.
\end{proof}

As a consequence of Theorem \ref{th:Zanella} and Corollaries \ref{cor:1} and \ref{cor:2} we completely determine $r$ for $r$-fat LP-polynomials.

\begin{theorem}\label{th:finalLP}
Let $\delta\in\mathbb{F}_{q^n}$, $\gcd(s,n)=1$, $f(x)=x+\delta x^{q^{2s}}\in\mathbb{F}_{q^n}[x]$ and $g(x)=x^{q^{s(n-1)}}+\delta x^{q^s}\in\mathbb{F}_{q^n}[x]$.
The polynomials $f(x)$ and $g(x)$ are $r$-fat polynomials over $\fqn$ of index $s$ and $0$ respectively, where 
\begin{itemize}
    \item $r=0$, if $\N_{q^n/q}(\delta)\ne 1$;
    \item $r=\frac{q^{n-1}-1}{q^2-1}$, if $\N_{q^n/q}(\delta)= 1$ and $n$ is odd;
    \item $r=\frac{q^2(q^{\frac{n-2}2}+1)(q^{\frac{n-2}2-1}-1)}{q^2-1}+1$, if $\N_{q^n/q}(\delta)= 1$, $n$ is even, and either $q$ is even or $4\mid n$;
    \item $r=\frac{(q^n/2+1)(q^{n/2-1}-1)}{q^2-1}$, if $\N_{q^n/q}(\delta)= 1$, $n$ is even, $q$ is odd, $2\mid n$ and $4\nmid n$.
\end{itemize}
\end{theorem}

\section{Open problems}\label{sec:open}

We conclude the paper with some open problems.

Up to now, there are only few examples of $r$-fat polynomials with $r>0$ and $r$ small with respect to $q$. Even for the case $r=1$ we do not know many examples. 

\begin{open}
    Find new examples of $1$-fat polynomials over $\fqn$.
\end{open}

%The only known exceptional $r$-fat polynomials of index $t$ are scattered.
%If $r>0$, then up to equivalence $r=1$ and $t=0$; see Corollary \ref{cor:solo1}. Remark \ref{RemarkSupercazzoloso} suggests a possible strategy to find such polynomials. 
%\begin{open}
%    Find exceptional $r$-fat polynomials with $r>0$, or prove that they do not exist.
%\end{open}
%\begin{open}
%    Find exceptional $1$-fat polynomials of index zero, or prove that they do not exist.
%\end{open}

In this paper we completely determine the integer $r$ for which LP-polynomials are $r$-fat. It could be interesting to investigate other families of binomials which contain scattered polynomials.
\begin{open}
    Consider the polynomial $f(x)=x^{q^s}+\delta x^{q^{2s}}\in \fqn[x]$ with $\delta \ne 0$. Which is the integer $r$ such that $f(x)$ is an $r$-fat polynomial of index $0$?
    Zanella in \cite[Proposition 3.8]{Zanella} proved that $r>0$ when $n=5$, and Montanucci (cf. \cite[Remark 3.5]{Zanella}) proved that $r>0$ for any $n>5$.
\end{open}

\begin{open}
    Consider the polynomial $f(x)=x^{q^s}+\delta x^{q^{s+n}}\in \mathbb{F}_{q^{2n}}[x]$ with $\delta \ne 0$.
    Which is the integer $r$ such that $f(x)$ is an $r$-fat polynomial of index $0$?
    In \cite{CMPZ}, the authors proved that the maximum weight of $f(x)$ as a polynomial of index $0$ is at most two (cf. \cite[Proposition 4.1]{CMPZ}) and that $r$ is a multiple of $\frac{q^n-1}{q-1}$ (cf. \cite[Corollary 5.4]{CMPZ}). For $n=6$ the papers \cite{BCsM} and \cite{PZ2019} characterize those $\delta \in \F_{q^6}$ such that $f(x)$ is scattered of index $0$, whereas in \cite[Theorem 1.1]{PZZ} it has been proved that $r\geq \frac{q^n-1}{q-1}$ for every $\delta \in \F_{q^{2n}}^*$, whenever $n\geq 4s+2$ if either $q=3$ and $s>1$, or $q=2$ and $s>2$; or $n\geq 4s+1$ otherwise.
\end{open}

The next open problem concerns the classification of $1$-fat polynomials over $\mathbb{F}_{q^4}$.
By Corollary \ref{cor:PG1q4clubs}, this classification is obtained once a description of shape $L_f$ is found for the linear set described in \eqref{eq:iiB}.

\begin{open}
Find a polynomial $f(x)\in\mathbb{F}_{q^4}$ such that $L_f$ is ${\rm P\Gamma L}$-equivalent to the $\fq$-linear set of $\PG(1,q^4)$ described in \eqref{eq:iiB}.
\end{open}

Our last open problem concerns the new example of $1$-fat polynomial $f_w(x)$ with maximum weight $3$ determined in Theorem \ref{th:3club}. For $q=4$, we found that $L_{f_w}$ is $\mathrm{P\Gamma L}$-equivalent to $L_\lambda$  as defined in \eqref{eq:exKMnopol}, for some $\lambda \in \mathbb{F}_{4^5}$. It would be interesting to investigate the equivalence issue for any power $q$ of $4$. 

\begin{open}
Prove or disprove that the linear set $L_{f_w}$ in Theorem \ref{th:3club} is $\mathrm{P\Gamma L}$-equivalent to $L_\lambda$ as defined in \eqref{eq:exKMnopol} for some $\lambda \in \mathbb{F}_{q^5}$.
\end{open}

\section{Acknowledgements} 
The research of the last three authors was supported by the Italian National Group for Algebraic and Geometric Structures and their Applications (GNSAGA - INdAM).
The third author is funded by the project ``Attrazione e Mobilità dei
Ricercatori'' Italian PON Programme (PON-AIM 2018 num. AIM1878214-2).
The third and the fourth authors are supported by the project ``VALERE: VAnviteLli pEr la RicErca" of the University of Campania ``Luigi Vanvitelli''.

\end{document}